\documentclass[a4paper,11pt]{amsart}
\usepackage{amssymb,amsmath,amsfonts,amscd,euscript,accents}
\usepackage[backref=page]{hyperref}
\usepackage{color} 

\usepackage{enumitem}

\usepackage{tikz, tikz-3dplot, pgfplots}
\usepackage{tikz-cd}

\numberwithin{equation}{section}

\newtheorem{theoremA}{Theorem}

\newtheorem{thm}{Theorem}[section]
\newtheorem{prop}[thm]{Proposition}
\newtheorem{lem}[thm]{Lemma}
\newtheorem{cor}[thm]{Corollary}
\newtheorem{rem}[thm]{Remark}

\newtheorem{example}[thm]{Example}
\newtheorem{dfn}[thm]{Definition}

\newcommand{\bC}{{\mathbb C}}

\newcommand{\bP}{{\mathbb P}}
\newcommand{\bQ}{{\mathbb Q}}
\newcommand{\bZ}{{\mathbb Z}}
\newcommand{\bN}{{\mathbb N}}
\newcommand{\eF}{{\EuScript{AF}}}
\newcommand{\msl}{\mathfrak{sl}}
\newcommand{\fh}{{\mathfrak h}}
\newcommand{\fb}{{\mathfrak b}}

\newcommand{\mb}{\mathbf}
\newcommand{\mr}{\mathrm}
\newcommand{\bdim}{\mb{dim}}

\pgfplotsset{compat=1.18} 
\begin{document}


\title[Laumon parahoric local models via quiver Grassmannians]
{Laumon parahoric local models via quiver Grassmannians}

\author{Evgeny Feigin}
\address{E. Feigin:\newline
School of Mathematical Sciences, Tel Aviv University, Tel Aviv
69978, Israel
}
\email{evgfeig@gmail.com}
\author{Martina Lanini}
\address{M. Lanini:\newline Dipartimento di Matematica\\ Universit\`a di Roma ``Tor Vergata'',  Via della Ricerca Scientifica 1, I-00133 Rome, Italy}
\email{lanini@mat.uniroma2.it}
\author{Alexander P\"utz}
\address{A. P\"utz:\newline
Institute of Mathematics\\
University Paderborn\\
Warburger Str. 100\\
33098 Paderborn\\
Germany}
\email{alexander.puetz@math.uni-paderborn.de}


\begin{abstract}
Local models of Shimura varieties in type A can be realized inside  products of Grassmannians via certain linear algebraic conditions. 
Laumon suggested a generalization which can be identified with a family over a line whose general fibers 
are quiver Grassmannians for the loop quiver and the special fiber is a certain quiver Grassmannian for the cyclic quiver. 
The whole family sits inside the Gaitsgory central degeneration of the affine Grassmannians. We study the properties of the special fibers of the (complex) Laumon local models for arbitrary parahoric subgroups in type A using the machinery of quiver representations.
We describe the irreducible components and the natural strata with respect to the group action for the quiver Grassmannians in question. 
We also construct a  cellular decomposition and provide an explicit description for the corresponding poset of cells.
Finally, we study the properties of the desingularizations of the irreducible components and show that
the desingularization construction is compatible with the natural projections between the parahoric subgroups.
\end{abstract}

\maketitle

\section{Introduction}
Local models of Shimura varieties play an important role in various problems of arithmetics and number theory (see \cite{PZ13,PRS13,P18,Zhu19, HR23}). 
Various methods for their study were developed and various links with other branches of mathematics were discovered. 
In particular, affine Grassmannians and affine flag varieties naturally show up in the theory (see \cite{Ga01,Zhu17,HR20}). Also, certain local models can be described via quiver Grassmannians  (see \cite{G01,G03,HZ23-1,HZ23-2,PZ22}). 
In this paper we are studying the special fibers of the Laumon local models \cite{HN02,Zho19} over the field of complex numbers (for arbitrary parahoric subgroups in type $A$)  via the theory of quiver Grassmannians. Let us describe our results.

Quiver Grassmannians are natural generalizations of the classical Grassmann varieties. To define a quiver Grassmannian ${\rm Gr}_d(M)$ one fixes a quiver $Q$
with set of vertices $Q_0$, a $Q$-module $M=(M_i)_{i\in Q_0}$ and a dimension
vector $d\in\bZ_{\ge 0}^{Q_0}$ (see e.g. \cite{CaRe2008,CI20,CEFR2018,Re13,Schiffler2014,Scho92}). If one is able to realize a projective algebraic variety as a quiver Grassmannian for a simple quiver, then one can exploit the machinery of quiver representations to study its various properties, like stratifications, cellular decompositions, group actions, degenerations, irreducible components, etc. 
(see \cite{CFR12,CFFFR17,CFFFR20,CL15,Pue2019,Pue2020} for some examples). 

In this paper we mostly deal with the cyclic quiver; for a positive integer $r$  we denote by $\Delta_r$ the cyclic quiver on $r$ vertices. Let us fix two positive integers $n$ and $\omega$ and a subset $S\subset [n]= \{1,\dots,n\}$ with $|S|=r$.
To this data we attach a representation $U_{n\omega,S}$  of the quiver $\Delta_r$ defined as follows.  Let $A_\infty$ be an infinite 
equioriented type $A$ quiver with the set of vertices labeled by $\bZ$. We consider the  $A_\infty$ module $M(a)$ of dimension $1$ at vertices from $a+1$ to $a+n\omega$ and zero otherwise with obvious maps between the spaces. 
The $\mod n$ projection $A_\infty\to \Delta_r$, where the vertices of $\Delta_r$ are labeled by the elements of the set $S$, induces a map between the $A_\infty$ modules and the $\Delta_r$ modules.
We define $U_{n\omega,S}$ as the image of the direct sum  $\bigoplus_{a=1}^n M(a)$. The paper is devoted to the study of the quiver Grassmannians 
\[
X_S(k,n,\omega)={\rm Gr}_{k\omega,\dots,k\omega}(U_{n\omega,S})
\]
for $1\le k <n$.
Special cases of these quiver Grassmannians were studied in \cite{FLP21,FLP23}.
The varieties $X_S(k,n,\omega)$ enjoy many nice properties and unexpected connections. In particular,  for $\omega=1$ and $S=\{1,\dots.n\}$ the topology and combinatorics of the quiver Grassmannians in question  are closely related to that of the totally nonnegative Grassmannians \cite{KLS13,Lam16,Lus98a, Pos06,Rie99,W05}. 
For general $S$ and $\omega$ the varieties $X_S(k,n,\omega)$ are isomorphic to the special fibers of the 
parahoric Laumon local models \cite{HN02,Zho19}. Using  representation theory of quivers and the lattice realization of the affine flag varieties one sees that
each quiver Grassmannian as above is isomorphic to a union of Schubert varieties inside the affine flags and that it is naturally included into a family of quiver Grassmannians with the general fiber isomorphic to 
a Schubert variety inside affine Grassmannians. We note also that the general fiber is identified with certain quiver Grassmannian for the loop quiver.

Now we formulate the main properties of $X_S(k,n,\omega)$. Here is the first theorem.

\begin{theoremA} (Theorem \ref{thm:geometric-properties})
The variety  $X_S(k,n,\omega)$ is equidimensional of dimension $\omega k(n-k)$. The irreducible components are parameterized by $I \in \binom{[n]}{k}$ such that $i \in I\setminus S$ implies $i+1 \in I$; the irreducible components are normal, Cohen-Macaulay and have rational singularities.
The variety $X_S(k,n,\omega)$ is equipped with a $\mathbb{C}^*$-action and the corresponding  Bialynicki-Birula-decomposition is cellular.
\end{theoremA}

Normality and Cohen-Macaulayness for a larger family of varieties were also shown in \cite{HR23} by other techniques (see also \cite{H13}).

For each pair $S'\subset S\subset [n]$ there exists a natural projection map ${\rm pr}_{S \setminus S'} : X_{S}(k,n,\omega)\to X_{S'}(k,n,\omega)$. We show that these maps are surjective. We also use the maps  
to define the action of the automorphism group ${\rm Aut}_{\Delta_r}(U_{\omega n,S})$ on $X_{S'}(k,n,\omega)$. We describe the automorphism groups explicitly and prove the following theorem, where 
$\mr{Aut}_{\Delta_n}(U_{\omega n})$ corresponds to the largest case $S=[n]$.

\begin{theoremA} (Theorem \ref{trm:cells-are-strata})
The varieties $X_S(k,n,\omega)$ are acted upon by an algebraic torus $T$ and each Bialynicki-Birula cell $C$ is torus-stable and contains exactly one torus-fixed point.
The $\mr{Aut}_{\Delta_n}(U_{\omega n})$-orbit of a torus fixed point coincides with the cell, while the corresponding stratum (isomorphism class in the quiver Grassmannian) coincides with the $\mr{Aut}_{\Delta_r}(U_{\omega n,S})$-orbit.
\end{theoremA}

We provide explicit description of the cells (in particular, via the juggling patterns) and describe the poset of structure.

Finally, we work out explicitly the desingularization construction of \cite{PuRe2022} in our special case. More precisely, for each
irreducible component of $X_S(k,n,\omega)$ we consider the quiver Grassmannian for the extended cyclic quiver which is smooth and birationally surjects onto the component. We denote by $\hat{X}_{S}(k,n,\omega)$ the union of all the desingularized components. We describe the properties of $\hat{X}_{S}(k,n,\omega)$; in particular, we prove the following theorem.  

\begin{theoremA} (Theorem \ref{thm:natural-projections-desing})
For any $S' \subset S\subset [n]$, the natural map $X_{S}(k,n,\omega)\to X_{S'}(k,n,\omega)$ admits a surjective lift $\hat{X}_{S}(k,n,\omega)\to \hat{X}_{S'}(k,n,\omega)$. 
\end{theoremA}

Our paper is organized in the following way. 
In Section \ref{QGandLM} we introduce background and basic definitions on quiver Grassmannians, parahoric Laumon local models. We discuss how local models are realized as  special quiver Grassmannians $X_S(k,n,\omega)$; the isomorphism with certain Schubert varieties within an appropriate partial flag variety is explained at the end of this section. 
Section \ref{AGprop} deals with geometric properties of $X_S(k,n,\omega)$, in particular, we prove here Theorems A and B, discussing torus actions, cellular decomposition and stratification. 
In Section \ref{sec:moment-graph+Cohomology} we describe the moment graph for the torus action on $X_S(k,n,\omega)$ and on the closure of any cell, investigating the closure inclusion relation of the cell poset. We also provide a formula for the Poincar\'e polynomial of $X_S(k,n,\omega)$. 
Finally, Section \ref{sec:desingularization} is about  the desingularization construction.

\section{Quiver Grassmannians and local models}\label{QGandLM}

\subsection{Basics on Quiver Representations and Quiver Grassmannians}
In this section we recall fundamental definitions from the representation theory of quivers. For more details see \cite{Schiffler2014,Cerulli2011}. A quiver $Q$ is a tuple $(Q_0,Q_1,s,t)$ where $Q_0$ is the set of vertices, $Q_1$ is the set of arrows between the vertices and $s,t : Q_1 \to Q_0$ map each arrow $a$ to its source $s_a$ and target $t_a$, respectively. 
A $Q$-representation $M$ consists of a tuple $(M^{(i)})_{i \in Q_0}$ of $\bC$-vector spaces over the vertices, and a tuple $(M_a )_{a \in Q_1}$ of linear maps 
\[ M_a : M^{(s_a)} \to M^{(t_a)}.\]

A morphism $\psi$ of $Q$-representations $M$ and $N$, or $Q$-morphism,  is a collection of linear maps $\psi_i : M^{(i)} \to N^{(i)}$ such that $\psi_{t_a} \circ M_{a} = N_a \circ \psi_{s_a}$ holds for all $a \in Q_1$.   By $\mathrm{Hom}_Q(M,N)$ we denote the set of all $Q$-morphisms from $M$ to $N$. We write $\mathrm{rep}_\bC(Q)$ for the category of  finite dimensional complex $Q$-representations. 

A subrepresentation $N \subseteq M$ is given by a tuple of vector subspaces $N^{(i)} \subset M^{(i)}$, such that $M_a(N^{(s_a)}) \subseteq N^{(t_a)}$  for all arrows $a\in Q_1$. The dimension vector $\textbf{dim} \, M \in \bZ^{Q_0}$ of $M \in \mathrm{rep}_\bC(Q)$ is the vector whose $i$-th entry is $ \dim_\bC  M^{(i)}$. 
\begin{dfn}
For $\mb{e} \in \bZ^{Q_0}$ and $M \in \mathrm{rep}_\bC(Q)$, the quiver Grassmannian $\mr{Gr}_{\mb{e}}(M)$ is the variety of all $\mb{e}$-dimensional subrepresentations of $M$.
\end{dfn}
For a point $U \in {\rm Gr}_{\bf e}(M)$ the isomorphism class 
$\mathcal{S}_U=\{V \in {\rm Gr}_{\bf e}(M):\ V\simeq U\}$ in the quiver Grassmannian is called stratum and is irreducible (cf. \cite[Lemma~2.4]{CFR12}). The automorphism group $\mr{Aut}_Q(M) \subset \mathrm{Hom}_Q(M,M)$ acts on $\mr{Gr}_{\mb{e}}(M)$ as
\[ A.\big(U^{(i)}\big)_{i \in Q_0} := \Big( A_i\big(U^{(i)}\big) \Big)_{i \in Q_0} \quad \mr{for} \  A \in \mr{Aut}_Q(M) \ \mr{and} \ U \in {\rm Gr}_{\bf e}(M).\]

\subsection{The main objects}
By $\Delta_n$ we denote the equioriented cyclic quiver on the vertex set $\bZ_n := \bZ/n\bZ$ with arrows $a: i \to i+1$  for all $i \in \bZ_n$. For $M \in \mr{rep}_\bC(\Delta_n)$ we write $M_i$ instead of $M_a$ for the map along the arrow $a: i \to i+1$. Now we introduce one particular class of $\Delta_n$-representations which is required for the definition of the main objects.

For any $m \in \bN$ and any subset $S \subseteq [n]$ of cardinality $r$, we define the $\Delta_r$-representation $U_{m,S}$ as follows: Take the vector spaces $M^{(i)}:=\bC^m$ for all $i \in \bZ_r$ and let $B^{(i)}:=\{v_j^{(i)} \, : \, j \in [m] \}$ be the standard basis of the $i$-th copy of $\bC^{m}$. Each map $M_i$ (evaluated on the standard basis) is described by an $m \times m$ matrix with entries 
\begin{equation}\label{eq:linear-maps}
    (M_i)_{k,l} := \delta_{k,l+s_{i+1}-s_i} = \begin{cases} 1 & {\rm if} \ k  = l+s_{i+1}-s_i,\\ 0 & {\rm otherwise}\end{cases}
\end{equation} 
where $S=\{s_i\mid i\in [r]\}$, with  $1\leq s_1< s_2<\dots < s_r\leq n$ and, by convention, $s_{r+1}:= s_1+n$.

\begin{dfn}
For $\omega \in \bN$, $k \in [n]$ and $S \subset [n]$ with $\#S=r$
we define the quiver Grassmannian
\[
X_S(k,n,\omega) := {\rm Gr}_{(k\omega,\dots,k\omega)}(U_{\omega n,S}).
\]
\end{dfn}
\begin{rem}
For $\omega=1$ and $\#S=1$ we obtain $X_S(k,n,\omega) = {\rm Gr}_k(\bC^n)$ whereas for $\omega>1$ and $\#S=1$ the variety $X_S(k,n,\omega)$ is a quiver Grassmannian for the representation of the loop quiver $\Delta_1$ with vector space $\bC^{\omega n}$ and non-trivial map described by the $\omega n \times \omega n$ matrix with entries $M_{k,l} = \delta_{k,l+n}$.
\end{rem}

\subsection{Degenerations}
The varieties $X_S(k,n,\omega)$ can be understood as (special fibers of) the parahoric Laumon local models (see \cite{HN02,Zho19}) over $\mathbb{C}$. Let us recall the construction.

We fix an $n$-dimensional vector space $\bC^n$, an auxiliary variable $t$ and consider the tensor product
$T_{n,\omega}=\bC^n\otimes\bC[t]/(t^\omega)$. In particular, $T_{n,\omega}$ admits a nilpotent action
of the operator of multiplication by $t$.
Next, let $F_\varepsilon$, $\varepsilon\in\mathbb{A}^1$, be a family of operators $F_\varepsilon$ on $T_{n,\omega}$ induced  by the following $n\times n$ matrix 
\begin{equation}\label{Feps}
\begin{pmatrix}
0 & 0 & \dots & 0 & 0 & t+\varepsilon \\
1 & 0 &  \dots & 0 & 0 &  0\\
0 & 1 & \dots  &  0 & 0 & 0\\
\vdots & \vdots & \ddots & \vdots & \vdots & \vdots\\
0 & 0 &   \dots & 1 & 0 & 0\\
0 & 0 &   \dots & 0 & 1 &  0
\end{pmatrix}.
\end{equation}
More precisely, $F_\varepsilon$ commutes with the multiplication by $t$ and the matrix
\eqref{Feps} represents the action of $F_\varepsilon$ on $\bC^n\simeq \bC^n\otimes 1\subset T_{n,\omega}$ 
in certain basis $b_i$, $i\in [n]$. 
Clearly,
\begin{equation}\label{F^n}
F_\varepsilon^n = (t+\varepsilon){\rm Id}. 
\end{equation}

Now let us fix a number $\overline k$, with $1\le \overline k< n\omega$ and a subset $S\subset [n]$,
$S=\{s_1<\dots <s_r\}$.
\begin{dfn}
The parahoric Laumon local model $M_{\overline k,n,\omega}(\varepsilon)$ consists of collections 
$(L_i)_{i\in S}$ of $\overline k$-dimensional subspaces of $T_{n,\omega}$ such that 
$F_{\varepsilon}^{s_{i+1}-s_i} L_{s_i}\subset L_{s_{i+1}}$ for all $i=1,\dots,r$
(here $s_{r+1}=s_1$).
\end{dfn}
\begin{lem}
One has $M_{k\omega,n,\omega}(0)\simeq X_S(k,n,\omega)$.    
\end{lem}
\begin{proof}
By definition, both $M_{k\omega,n,\omega}(0)$ and $X_S(k,n,\omega)$ are quiver Grassmannians of submodules of dimension $(k\omega,\dots,k\omega)$ inside certain representations of  the cyclic quiver on $n$ vertices. Hence we only need to identify these quiver representations. 
The identification is provided by the map
$b_i\otimes t^j\mapsto v^{(\bullet)}_{i+nj}$ (independently of the upper index $(\bullet)$).   
\end{proof}

In order to describe the general fiber we denote by  $\overline T_{n,\omega}$
a representation of the (one vertex) one loop quiver with the underlying vector space being
$T_{n,\omega}$ and the loop acting via the multiplication by $t$.

The following lemma was proved in \cite{HN02}. We briefly recall the statement and the proof for
the reader's convenience.
\begin{lem}
For a nonzero $\varepsilon$ one has
\[
M_{k\omega,n,\omega}(\varepsilon)\simeq {\rm Gr}_{k\omega}(\overline T_{n,\omega}).
\]
\end{lem}
\begin{proof}
The first observation is that for  $\varepsilon\not= 0$ the map $F_\varepsilon$ is
invertible and hence a point $(L_i)_{i\in S}\in M_{k\omega,n,\omega}(\varepsilon)$ is completely determined by,
say, $L_{s_1}$. Hence one gets an embedding $M_{k\omega,n,\omega}(\varepsilon)\to {\rm Gr}(k\omega,T_{n,\omega})$.
The second observation is that, thanks to  \eqref{F^n},  each $L_i$ is invariant with respect to the multiplication by $t$. Finally, one shows that any $t$-invariant subspace $L$ shows
up  as $L_{s_1}$ for a point $(L_i)_{i\in S}\in M_{k\omega,n,\omega}(\varepsilon)$.
\end{proof}

\begin{rem}
The varieties ${\rm Gr}_{k\omega}(\overline T_{n,\omega})$ are members of a larger family of quiver Grassmannians for the loop quiver studied in \cite{FFR17}. To be precise, in the notation of \cite[Definition 3.1]{FFR17}:
\[
{\rm Gr}_{k\omega}(\overline T_{n,\omega})\simeq X^{(\omega)}_{k\omega,n}.
\]
\end{rem}

\begin{rem}
The variety ${\rm Gr}_{k\omega}(\overline T_{n,\omega})$ is irreducible of dimension $k\omega(n-k)$. It is isomorphic to a Schubert variety in the affine Grassmannian  for the affine $GL_n$ (see the discussion below). 
In other words, ${\rm Gr}_{k\omega}(\overline T_{n,\omega})$ can be realized as a Schubert variety inside a partial affine flag variety for $SL_n$, corresponding to a maximal parabolic subgroup. 
\end{rem}

\begin{cor}\label{family}
There exists a family over ${\mathbb A}^1$ whose general fiber is  ${\rm Gr}_{k\omega}(\overline T_{n,\omega})$ and the special fiber is $X_S(k,n,\omega)$.    
\end{cor}

\subsection{Affine Grassmannians and flag varieties}
We start by recalling the lattice realization of affine Grassmannians. Let $\ell\in\mathbb{Z}$, then $\mathcal{AG}r^{(\ell)}$ is the locus of $t$-invariant subspaces $V\subset \bC^N\otimes \bC[t,t^{-1}]$
having the property that there exists an $N\in \mathbb{Z}_{>0}$ such that
\[
\bC^n\otimes t^{N}\bC[t]
 \subset V\subset \bC^n \otimes t^{-N}\bC[t], \hbox{ and } 
 \dim V/\bC^n\otimes t^N\bC[t]=\ell+Nn.
\]
The above ind-variety can be embedded into the projectivization of the space of semi-infinite wedges of charge $\ell$. This  is the infinite dimensional vector space spanned by vectors of the form
\begin{equation}\label{monom}
b_{i_1} t^{j_1} \wedge b_{i_2} t^{j_2} \wedge b_{i_3} t^{j_3}\wedge \dots 
\end{equation}
such that 
\begin{itemize}
\item $j_\bullet\le j_{\bullet +1}$ and if $j_\bullet = j_{\bullet +1}$, then $i_\bullet < i_{\bullet +1}$;
\item the wedge product \eqref{monom} contains a tail of the form 
\[
b_1t^N\wedge b_2t^N\wedge\dots\wedge b_nt^N\wedge b_1t^{N+1}\wedge b_2t^{N+1}\wedge\dots;  
\]
\item the number of factors in \eqref{monom} not contained in the tail is equal to $\ell+Nn$.
\end{itemize}

In the rest of this section we discuss the identification of the varieties ${\rm Gr}_{k\omega}(\overline T_{n,\omega})$ with certain Schubert varieties inside an affine Grassmannian for $GL_n$
(or, in other words, inside the partial affine flag varieties for $SL_n$ corresponding to maximal parabolic subalgebras). Before giving a general
construction, let us provide an example.

Let $n=2$, $k=1$ and $\omega\ge 1$. Let $P_0$ and $P_1$ be the maximal parabolic
subgroups of the affine Kac-Moody Lie group $\widehat{SL}_2$. In particular, the affine
Grassmannian is isomorphic to $\widehat{SL}_2/P_0$ and the affine Grassmannian for
$GL_2$ is the union of $\widehat{SL}_2/P_0$ and $\widehat{SL}_2/P_1$. The affine
Weyl group is isomorphic to the semi-direct product  $\bZ/2\ltimes \bZ$; the torus
fixed points in $\widehat{SL}_2/P_0$ are labeled by the elements $e\times \bZ$ and 
the torus
fixed points in $\widehat{SL}_2/P_1$ are labeled by the elements $\sigma\times \bZ$,
where $\bZ_2=\{e,\sigma\}$. For an affine Weyl group element $w$ let 
$Y_w\subset \widehat{SL}_2/P_0 \sqcup \widehat{SL}_2/P_1$ be the corresponding affine
Schubert variety. 

\begin{rem}
We are only interested in the Schubert varieties invariant under the action of the whole current group
$SL_2(\bC[t])$ (as opposed to the smaller Iwahori subgroup). These Schubert varieties correspond to
the elements inside $\bZ/2\times \bZ_{\ge 0}$. In particular, $\dim Y_{(e,m)}=2m$, $\dim Y_{(\sigma,m)}=2m+1$. 
\end{rem}

Recall that the quiver Grassmannians ${\rm Gr}_{\omega}(\overline T_{2,\omega})$ are naturally acted
upon by the current group $SL_2(\bC[t])$. Then one has the $SL_2(\bC[t])$ equivariant isomorphisms 
${\rm Gr}_{\omega}(\overline T_{2,\omega})\simeq Y_{(e,\omega/2)}$ for even $\omega$ and 
${\rm Gr}_{\omega}(\overline T_{2,\omega})\simeq Y_{(e,(\omega-1)/2)}$ for odd $\omega$.

Now let us formulate the general construction.
Let $\fh\subset\msl_n$ be the Cartan subalgebra. 
Let $\pi_1,\dots,\pi_{n-1}\in\fh^*$ be the set of fundamental weights. We use the standard notation $P$ and $Q$ for the weight and root lattices of $\msl_n$. 

 Let $P_0,\dots,P_{n-1}$ be the maximal parahoric subgroups of $\widehat{SL}_n$. In particular, $P_0$ is the current group $SL_n(\bC[t])$. 
Our goal is to embed  ${\rm Gr}_{k\omega}(\overline T_{n,\omega})$ inside the disjoint union of the (generalized) affine Grassmannians $\widehat{SL}_n/P_i$ \cite{Kum02}. 

The affine Kac-Moody Lie algebra $\widehat{\msl}_n$ is equal to $\msl_n\otimes\bC[t,t^{-1}]\oplus\bC K\oplus\bC d$, where $K$ is central and $d$ is the degree element. The affine Cartan subalgebra is given by 
$\widehat{\fh}=\fh\otimes 1\oplus\bC K\oplus\bC d$. For an element $\Lambda\in \widehat{\fh}^*$ its level is the value $\Lambda(K)$. 

Let $\Lambda_i$, $i=0,\dots,n-1$ be the fundamental (level one) weights for $\widehat{\msl}_n$ (see \cite{Kac85}).  In particular, the restrictions of $\Lambda_i$, $1\le i\le n-1$ to the (finite) Cartan subalgebra $\fh$ is equal to $\pi_i$ and  $\Lambda_0|_{\fh}=0$.  

We denote by $L(\Lambda_i)$ the  integrable irreducible highest weight $\widehat{\msl}_n$ representations of fundamental weights. In particular, these are all integrable irreducible level one  modules. 
Let $W$ and $\widehat{W}$ be the Weyl groups for $\msl_n$ and $\widehat{\msl}_n$. In particular,  
$\widehat{W}=W\ltimes Q$, where $Q$ is the root lattice for $\msl_n$ (here we do not distinguish between the root and coroot lattices). The affine Weyl group acts on the weights, we will be mostly interested in the $\widehat{W}$-shifts of the fundamental weights. 

Recall that each $L(\Lambda_i)$ is a weight module with respect to the affine Cartan subalgebra. In particular, for any $\sigma\in \widehat{W}$ the dimension of the weight 
$\sigma(\Lambda_i)$ vectors inside $L(\Lambda_i)$ is equal to one. We choose a non-zero weight $\sigma(\Lambda_i)$ vector inside 
$L(\Lambda_i)$ and denote it by $v_{\sigma(\Lambda_i)}$.

The following claim is well-known and is very useful (see e.g. \cite{FL06}): 
for any $\mu\in P$ there exists an $i=0,\dots,n-1$ and $\sigma\in \widehat{W}$ such that $\mu=\sigma(\Lambda_i)|_{\fh}$.

Recall that one has the $\widehat{SL}_n$-equivariant embeddings $\widehat{SL}_n/P_i\subset \bP(\Lambda_i)$ (in particular, the affine Grassmannian sits inside the projectivization of the basic representation $L(\Lambda_0)$). Thus, all the lines 
$[v_\mu]$, $\mu\in P$ belong to the disjoint union of the generalized affine Grassmannians. By definition, the set of $[v_\mu]$ coincides with the set of all torus fixed points.

For $\sigma\in \widehat{W}$ let $Y_\sigma\subset \widehat{SL_n}/P_i$ be the Schubert subvariety, i.e. the closure of the Iwahori group orbit. 
We note that if $\sigma(\Lambda_i)|_{\fh}\in P_-$, then
$Y_\sigma$ is stable under the  $SL_n(\bC[t])$ action, i.e. 
$Y_\sigma=\overline{SL_n(\bC[t]).[v_\mu]}$.
We put forward the following proposition. 

\begin{prop}\label{Schubert}
Let $\sigma\in \widehat{W}$ be an element such that $\sigma(\Lambda_i)_{\fh}=-\omega\pi_{n-k}$ for some $i$.
Then ${\rm Gr}_{k\omega}(\overline T_{n,\omega})$ is isomorphic to $Y_\sigma$.     
\end{prop}
\begin{proof}
Let $\ell = (k\omega \mod n)$, $\ell=0,\dots,n-1$. We write $k\omega = ns +\ell$.

 Recall that $T_{n,\omega}$ is isomorphic to a tensor product $\bC^n\otimes \bC[t]/t^\omega$. 
 Let $b_1,\dots,b_n$ be a basis of $\bC^n$. 
 A point  $U\in {\rm Gr}_{k\omega}(\overline T_{n,\omega})$ defines (in an obvious way) a $k\omega$-dimensional subspace $\overline U\subset \bC^n\otimes {\rm span}\{t^i\colon i=0,\dots,\omega-1\}$.
We define a subspace $L_U\subset \bC^n\otimes \bC[t,t^{-1}]$ by
\[
L_U=t^{s-\omega}\overline U \oplus \bC^n\otimes t^s\bC[t].
\]
Since $U$ is $t$-invariant, so is $L_U$. 

By definition, $\dim L_U/\bC^n\otimes t^s\bC[t]=k\omega=sn+\ell$. 
 We conclude that $L_U$ belongs to $\mathcal{A}Gr^{(\ell)}$ and hence one gets an embedding ${\rm Gr}_{k\omega}(\overline T_{n,\omega})\subset \mathcal{A}Gr^{(\ell)}$. Moreover, one sees that ${\rm Gr}_{k\omega}(\overline T_{n,\omega})$ is the closure of the $SL_n(\bC[t])$ orbit of the point 
\[
{\rm span}\{b_1,\dots,b_k\}\otimes t^{s-\omega}\bC[t]\oplus {\rm span}\{b_{k+1},\dots,b_n\}\otimes t^{s}\bC[t].
\]
Recall the embedding  $\mathcal{A}Gr^{(\ell)}$ into the projective space  of the semi-infinite forms of charge $\ell$. Then the weight of the corresponding point in the space of forms is equal to $\omega\pi_k$. Hence in order to identify ${\rm Gr}_{k\omega}(\overline T_{n,\omega})$ with a Schubert variety $Y_\sigma$ we need to find a lowest weight vector in  a representation with highest weight $\omega\pi_k$; this lowest weight 
is equal to $-\omega\pi_{n-k}$.
\end{proof}

\begin{cor}
The varieties ${\rm Gr}_{k\omega}(\overline T_{n,\omega})$ exhaust the (disjoint) union of generalized affine Grassmannians, i.e.
\[
\bigcup_{\substack{\omega\ge 1\\ 1\le k\le n-1}} {\rm Gr}_{k\omega}(\overline T_{n,\omega}) = \bigsqcup_{i=0}^{n-1}
\widehat{SL}_n/P_i.\]
\end{cor}
\begin{proof}
 We note that for any weight $\mu\in P$ there exists an $\omega\ge 1$ and a fundamental weight $\pi_k$ such that 
 $-\omega\pi_k$ is less than $\mu$ (in the dominant order). For such $\omega$ and $k$ the Schubert variety $Y_\mu$ sits inside ${\rm Gr}_{k\omega}(\overline T_{n,\omega})$ (thanks to Proposition \ref{Schubert}). Now our Corollary follows
 from the fact that any partial flag variety is a union of its Schubert subvarieties.  
\end{proof}

Let $S\subset[n]=(\Delta_n)_0$ and let $P_S\subset \widehat{SL}_n$ be the parahoric subgroup defined by the set of simple roots $\alpha_i$, $i\in S$ (here we identify the set of vertices of the cyclic quiver with the vertices of the Dynkin diagram of type $A_{n-1}^{(1)}$). The proof of the following result relies on Theorem \ref{thm:natural-projections}, whose proof will be given in the following section.
\begin{lem}
 For each $S\subset[n]$ the quiver Grassmannian $X_S(k,n,\omega)$ is isomorphic to a union of Schubert subvarieties in the partial flag variety $\widehat{SL}_n/P_S$. The union of all $X_S(k,n,\omega)$ equals to the whole partial flag variety.
 \end{lem}
\begin{proof}
The claim of the lemma is true for $S$ being the whole set of vertices of $\Delta_n$ (see \cite{FLP23}). 
Now the general case is an immediate corollary from Theorem \ref{thm:natural-projections}.    
\end{proof}

\begin{rem}
Due to the lemma above each $X_S(k,n,\omega)$ is isomorphic to a union of affine Schubert varieties inside $\widehat{SL}_n/P_S$.
Hence the limit of the degenerations ${\rm Gr}_{k\omega}(\overline T_{n,\omega})\rightsquigarrow X_S(k,n,\omega)$ when $\omega$ goes to infinity gives the so called  central degeneration of affine Grassmannians into affine flag varieties (see \cite{Ga01,Zho19}).   
\end{rem}

\section{Algebro-geometric properties}\label{AGprop}

\subsection{Natural projections}\label{sec:natural-projections} 

Consider the maps
\[ {\rm pr}_j: \prod_{i \in \bZ_n} {\rm Gr}_{\omega k}(\bC^{\omega n}) \to \prod_{i \in \bZ_n \setminus \{j\} }{\rm Gr}_{\omega k}(\bC^{\omega n})  \quad {\rm for} \ j \in \bZ_n. \]
Let $T \subset [n]$ with $\#T = q$ and define ${\rm pr}_T := {\rm pr}_{t_1} \circ \hdots \circ {\rm pr}_{t_q}$.
\begin{rem}
For $T \subset [n]$ with $\#T = q$ and $r:=n-q$ the map ${\rm pr}_T$ induces the functor $\mr{pr}_T : \mr{rep}_\bC \Delta_n \to \mr{rep}_\bC \Delta_r$ where an object $M \in \mr{rep}_\bC \Delta_n$ is sent to 
\[ \Big( \big(M^{(s_i)}\big)_{i \in \bZ_r}, \big(M_{s_{i+1}-1}\circ \dots \circ M_{s_i}\big)_{i \in \bZ_r} \Big).\]
Here $s_1 < s_2 < \hdots < s_r$ denote the elements of $[n] \setminus T$.
\end{rem}
\begin{thm}\label{thm:natural-projections}
For $\omega \in \bN$, $k \in [n]$ and $S,S' \subset [n]$ with $S' \subset S$, the projections ${\rm pr}_j$ induce the surjective map 
\[{\rm pr}_{S \setminus S'} : X_{S}(k,n,\omega)\to X_{S'}(k,n,\omega). \]
\end{thm}
For the proof of this statement we have to understand the geometry of the varieties $X_S(k,n,\omega)$; the proof is given in Section~\ref{sec:geometric-properties}.
\subsection{Torus actions}\label{sec:torus-action}
Let $\lambda \in \bC^*$ act on the basis of the vector spaces of $U_{\omega n,S}$ with the weights $\mr{wt}(v_t^{(i)}):=t$ for all $i \in \bZ_r$ and $t \in [\omega n]$. This action extends to $X_S(k,n,\omega)$ by \cite[Lemma~1.1]{Cerulli2011}. 

The above $\bC^*$-action coincides with a cocharacter of the algebraic torus $T:=(\bC^*)^{n+1}$ acting on the basis of the vector spaces of $U_{\omega n,S}$ via 
\[ \gamma.v_t^{(i)} =  \gamma_0^{t-1} \gamma_{i-t+1}v_t^{(i)}  \quad \mr{for} \ i \in \bZ_r \ \mathrm{and} \ \big(\gamma_0,(\gamma_j)_{j\in \bZ_n}\big) \in T.\] 
\begin{rem}
This action extends to $X_S(k,n,\omega)$ by \cite[Lemma~5.12]{LaPu2020} and commutes with the projection maps.
\end{rem}
\begin{lem}\label{lem:fixed-points-coincide}
The fixed point set $X_S(k,n,\omega)^{\bC^*}$ is finite and it coincides with the fixed point set $X_S(k,n,\omega)^T$.
\end{lem}
\begin{proof}
The $\bC^*$-action constructed above satisfies the assumptions of \cite[Theorem~1]{Cerulli2011}. Hence this statement implies that the cardinality of the $\bC^*$-fixed point set is finite. There is an embedding of $\bC^*$ into $T$ given by $\gamma_0:=\lambda$ and $\gamma_j := \lambda$ for all $j \in \bZ_n$ preserving the $\bC^*$-action. Thus, both fixed point sets coincide since $X_S(k,n,\omega)^T \subseteq X_S(k,n,\omega)^{\bC^*}$ (cf. \cite[Remark~1.13]{LaPu2020}).
\end{proof}
There is a combinatorial description of the $T$-fixed points of $X_S(k,n,\omega)$. Let $\binom{[n]}{k}$ denote the set of all $k$-element subsets of $[n]$. 

\begin{dfn}
For $\omega \in \bN$, $k \in [n]$ and $S \subset [n]$ with $\# S =:r$, let ${\mathcal Jug}_S(k,n,\omega)$ denote the set
\[  \Big\{ (J_i)_{i \in \bZ_r} \in \prod_{i  \in \bZ_r} \binom{[\omega n]}{ k \omega} : 
\tau_{s_{i+1}-s_{i}}(J_{i}\cap \bZ_{\le n\omega-s_{i+1}+s_i}) \subset J_{i+1} \  \mr{for} \ \mr{all} \ i \in \bZ_r
\Big\}\] 
where $1\leq s_1< s_2<\dots < s_r\leq n$ with $s_{r+1}:= s_1+n$, and $\tau_q : \bZ \to \bZ, \ j \mapsto j+q$.
\end{dfn}

\begin{lem}\label{lem:T-fixed-point-parametrisation}
For $\omega \in \bN$, $k \in [n]$ and $S \subset [n]$ with $\# S =:r$, the set $X_S(k,n,\omega)^T$ is in bijection with ${\mathcal Jug}_S(k,n,\omega)$.
\end{lem}
\begin{proof}
The $\bC^*$-fixed points (resp. $T$-fixed) of ${\rm Gr}_{k\omega}(\bC^{\omega n})$ are in bijection with the sets from $\binom{[\omega n]}{ k \omega}$. Namely, they are the vector spaces spanned by 
$v_j$ for $j \in J \in \binom{[\omega n]}{ k \omega}$ where the $v_j$'s denote the standard basis vectors of $\bC^{\omega n}$. Hence a $T$-fixed point of $X_S(k,n,\omega)$ is parameterized by an $r$-tuple $(J_i)_{i\in \bZ_r}$ of index sets $J_i \in \binom{[\omega n]}{ k \omega}$ such that 
\[M_{i} \Big( \big\langle v^{(i)}_k :  k \in J_{i} \big\rangle \Big) \subset \big\langle v^{(i+1)}_l :  l \in J_{i+1}  \big\rangle \quad \hbox{for all} \ i \in \bZ_r, \]
where as usual $v^{(i)}_j$ denotes the $j$-th standard basis vector in the $i$-th copy of $\bC^{\omega n}$. By the explicit description of the maps $M_i$, the above conditions are equivalent to the conditions in the definition of ${\mathcal Jug}_S(k,n,\omega)$.
\end{proof}

\subsection{Cellular decomposition and stratification}\label{sec:geometric-properties}
In this section we use the decomposition of $U_{\omega n,S}$ into indecomposable $\Delta_r$-representations 
to describe geometric properties of the varieties $X_S(k,n,\omega)$.

Let $A_\infty$ be the equioriented type $A$ quiver of infinite length.
By $V(i,j)$ we denote the indecomposable $A_\infty$-representation with vector spaces  $V(i,j)^{(k)}=\mathbb{C}$ for any $k\in [i,j]$ and maps $V(i,j)_{k}=\textrm{id}_{\mathbb{C}}$ for any $k\in[i,j-1]$. All other maps and vector spaces are zero.
Let $F: A_\infty \to \Delta_n$ send $k$ to $k \mod n$, and $(a:k \to k+1)$ to $(\overline{a}: k \mod n \to k+1 \mod n)$. This induces the $\Delta_n$-representation $U_i(\ell)$,
 with vector spaces $U_i(\ell)^{(j)}:=\bigoplus_{k\in F^{-1}(j)}V(i+1,i+\ell)^{(k)}$ for any $j\in \bZ_n$ and obvious linear maps. We use the notation $U(i;\ell):=U_{i-\ell}(\ell)$ if we want to emphasize that the indecomposable representation ends over the vertex $i \in \bZ_n$.

\begin{prop}
For $\omega \in \bN$, $k \in [n]$ and $S \subset [n]$ with $\# S =:r$, there is an isomorphism of $\Delta_n$-representations: 
\[ U_{\omega n,S} \cong \bigoplus_{i \in \bZ_r} U(i;\omega r) \otimes \bC^{q_i} \]
where $1\leq s_1< s_2<\dots < s_r\leq n$, $s_{r+1}:= s_1+n$ and $q_i := s_{i+1} -s_i$.
\end{prop}
\begin{proof}
Both representations have dimension vector $(\omega n, \dots, \omega n) \in \bZ^r$ since the sum of all $q_i$ equals $n$ by construction. Let us index the basis of the representation on the right hand side in the following way:

The basis vector of $\bC$ in the $k$-th position of the summand $U(i;\omega r)$ for $s \in [q_i]$ is represented by the $(i+k+s-2)$-nd basis vector of $\bC^{\omega n}$ over the $(i+k-1)$-st vertex of $\Delta_r$. Now the linear maps (of the representation on the RHS) in this basis are represented by the matrices $M_j$ for $j \in \bZ_r$ as defined in Equation~\eqref{eq:linear-maps}.  
\end{proof}
\begin{rem}\label{nilp}
Let $\mr{rep}_\bC^{N}\Delta_r$ be the category with objects $M \in \mr{rep}_\bC\Delta_r$ such that 
\[ M_{i+N-1} \circ  M_{i+N-2} \circ \dots \circ M_{i+1}\circ M_i = 0 \ \mr{for} \ \mr{all} \ i \in \bZ_r.\]
The projective and injective objects of $\mr{rep}_\bC^{N}\Delta_r$ are exactly the $U_i(N)$ where each $U_i(N)$ is projective and injective.
\end{rem}
\begin{dfn}
$M \in \mr{rep}_\bC\Delta_r$ is nilpotent if there exists an integer $N \in \bN$ such that $M \in \mr{rep}_\bC^{N}\Delta_r$.    
\end{dfn}
The $\bC^*$-action on $X:=X_S(k,n,\omega)$ induces the BB-decomposition (cf. \cite{Birula1973}):
\begin{equation}\label{eqn:BB-decomposition}
X=\bigcup_{p \in X^{\bC^*}} W_p,  \quad \quad  \mathrm{with} \quad \quad  W_p := \left\{ x \in X : \lim_{z \to 0} z.x =p \right\}.
\end{equation}
where $W_p$ is called attracting set of $p$. 
\begin{lem}\label{lem:cell-decomp}
The BB-decomposition of $X_S(k,n,\omega)$ is cellular.
\end{lem}
\begin{proof}
The weights of the $\bC^*$-action from Section~\ref{sec:torus-action} coincide with the weights in \cite[Section~4.4]{Pue2020}. Hence the pieces of the BB-decomposition are cells by \cite[Theorem~4.13]{Pue2020}.
\end{proof}
Before we describe the structure of ${\rm Aut}_{\Delta_r}(U_{\omega n,S})$
we recall the following result which is a special case of \cite[Proposition~2.10]{FLP23}.
\begin{prop} \label{Aut}
The elements of the automorphism group ${\rm Aut}_{\Delta_n}(U_{\omega n})$
 are exactly the matrix tuples $ A = (A_i)_{i \in \mathbb{Z}_n}$ with
\[ 
A_i = \begin{pmatrix}
a^{(i)}_{1,1} & & & &\\
a^{(i)}_{2,1}& a^{(i-1)}_{1,1} & & &\\
\vdots & \vdots & \ddots& &\\
a^{(i)}_{\omega n-1,1} & a^{(i-1)}_{\omega n-2,1} & \hdots & a^{(i-\omega n+2)}_{1,1} &\\
a^{(i)}_{\omega n,1} & a^{(i-1)}_{\omega n-1,1} & \hdots & a^{(i-\omega n-2)}_{2,1} & a^{(i-\omega n+1)}_{1,1}
\end{pmatrix} 
\]
where $a^{(i)}_{k,1} \in \mathbb{C}$ for all $i \in \mathbb{Z}_n$, $k \in [2,\omega n]$ and $a^{(i)}_{1,1} \in \mathbb{C}^*$ for all $i \in \mathbb{Z}_n$. In particular, $\dim_\bC {\rm Aut}_{\Delta_n}(U_{\omega n})=\omega n^2$.
\end{prop}
Let $\mathcal{T}_q$ denote the $m \times m$ matrix with entries 
\begin{equation}
 \big(\mathcal{T}_q\big)_{k,l} := \delta_{k,l+q} = \begin{cases} 1 & {\rm if} \ k  = l+q,\\ 0 & {\rm otherwise}\end{cases} \quad \mathrm{for} \ k,l \in [m]. 
\end{equation} 
\begin{prop}
Let $\omega \in \bN$, $k \in [n]$ and $S \subset [n]$ with $\# S =:r$,  $1\leq s_1< s_2<\dots < s_r\leq n$, $s_{r+1}:= s_1+n$ and $q_i := s_{i+1} -s_i$. The elements of the automorphism group ${\rm Aut}_{\Delta_r}(U_{\omega n,S})$
 are matrix tuples $ A = (A_i)_{i \in \mathbb{Z}_r}$ with lower triangular block matrices 
\[ 
A_i = \begin{pmatrix}
(a^{(i)}_{1,k,l}) & & & &\\
(a^{(i)}_{2,k,l})& (a^{(i-1)}_{1,k,l}) & & &\\
\vdots & \vdots & \ddots& &\\
(a^{(i)}_{\omega r-1,k,l}) & (a^{(i-1)}_{\omega r-2,k,l}) & \hdots & (a^{(i-\omega r+2)}_{1,k,l}) &\\
(a^{(i)}_{\omega r,k,l}) & (a^{(i-1)}_{\omega r-1,k,l}) & \hdots & (a^{(i-\omega r-2)}_{2,k,l}) & (a^{(i-\omega r+1)}_{1,k,l})
\end{pmatrix} 
\]
The diagonal blocks $(a^{(i)}_{1,k,l})$ are invertible and of size $q_i \times q_i$ and the blocks $(a^{(i)}_{j,k,l})$ below the diagonal are of size $q_i \times q_{i-j+1}$. 
In particular
\[ \dim_\bC {\rm Aut}_{\Delta_r}(U_{\omega n,S})= \sum_{i=1}^{r} \sum_{j=1}^{\omega r} q_i\cdot q_{i-j+1}. \]
\end{prop}
\begin{proof}
First we compute the endomorphism algebra ${\rm End}_{\Delta_r}(U_{\omega n,S})$. By definition it holds that $(E_i)_{i\in\bZ_r} \in {\rm End}_{\Delta_r}(U_{\omega n,S})$ if and only if
\[ E_{i+1} \mathcal{T}_{q_i} = \mathcal{T}_{q_{i+1}} E_i \quad \mathrm{for} \ \mathrm{all} \ i \in \mathbb{Z}_r. \]
This is equivalent to 
\begin{equation}\label{eqn:AutGp}
    E_{i+1} \mathcal{T}_{q_i}(v_l^{(i)}) = \mathcal{T}_{q_{i+1}} E_i(v_l^{(i)}) \quad \mathrm{for} \ \mathrm{all} \ i \in \mathbb{Z}_r, \, l \in [\omega n].
\end{equation}  
From these equations we obtain the claimed shape of the lower triangular block matrices. Imposing the condition that the matrices should be invertible we obtain the desired description of the automorphisms.
\end{proof}
\begin{rem}
    The projections ${\rm pr}_j$ induce the group homomorphism
\[{\rm pr}_{S \setminus S'} :  {\rm Aut}_{\Delta_r}(U_{\omega n,S}) \to {\rm Aut}_{\Delta_{r'}}(U_{\omega n,S'}). \]
This induces an $\mr{Aut}_{\Delta_n}(U_{\omega n})$-action on $X_S(k,n,\omega)$. For the orbits of this induced action we write $\mr{Aut}_{\Delta_n}(U_{\omega n})$-orbits for short.
\end{rem}
\begin{thm}\label{trm:cells-are-strata}
For $\omega\in \bN$, $k \in [n]$ and $S \subset [n]$, each cell $C \subset X_S(k,n,\omega)$ is $T$-stable and contains exactly one $T$-fixed point $P_C$. The $\mr{Aut}_{\Delta_n}(U_{\omega n})$-orbit of $P_C$ coincides with $C$ whereas the stratum of $P_C$ coincides with the $\mr{Aut}_{\Delta_r}(U_{\omega n,S})$-orbit of $P_C$ and contains $C$.
\end{thm}
\begin{proof}
The explicit description of the cells from Lemma~\ref{lem:cell-decomp} together with Lemma~\ref{lem:fixed-points-coincide} implies that each cell is $T$-stable and contains exactly one $T$-fixed point. The $\Delta_r$-representation $U_{\omega n,S}$ is injective. Hence \cite[Lemma~2.28]{Pue2019} implies that strata are $\mr{Aut}_{\Delta_r}(U_{\omega n,S})$-orbits.  The projection $\mr{pr}_{[n] \setminus S}$ is surjective by Theorem~\ref{thm:natural-projections} and, by $T$-equivariance, sends fixed points to fixed points and cells to cells. Thus the $\mr{Aut}_{\Delta_n}(U_{\omega n})$-orbits coincide with the cells.
\end{proof}

\begin{rem}
    For the maximal covering set $[n]$ of $S$ we obtain the cellular decomposition of the induced Aut-action. With the same arguments as above we can prove that the induced Aut-action for any covering set of $S$ provides a refinement of the stratification of $X_S(k,n,\omega)$ into the $\mr{Aut}_{\Delta_r}(U_{\omega n,S})$-orbit.
\end{rem}
\begin{rem}\label{rem:more-fixed-points-than-strata}
In contrast to the setting of \cite{FLP23}, there exist strata which contain more than one $T$-fixed point if $\#S < n$.
\end{rem}
\begin{example}
Let $n=2$, $k=\omega=1$ and $S=\{1\}$. Then $U_2=U_1(2) \oplus U_2(2)$. The image of $U_2$ under the projection $\mr{pr}_2$ is $U_{2,\{1\}} = U_1(1)\otimes \bC^2$ and $X_{\{1\}}(1,2,1)\cong \mr{Gr}_1(\bC^2)$ contains exactly two $\bC^*$-fixed points which are both isomorphic to $U_1(1)$ as $\Delta_1$-representations. 
\end{example}
\begin{rem}
The decomposition into $\mr{Aut}_{\Delta_n}(U_{\omega n})$-orbits is a refinement of the decomposition into $\mr{Aut}_{\Delta_n}(U_{\omega n,S})$-orbits.
\end{rem}
\subsection{Proof of Theorem~\ref{thm:natural-projections}}\label{sec:natural-projections-are-surective}
We apply the result from the previous sections to prove sujectivity of the projection
\[{\rm pr}_{S \setminus S'} : X_{S}(k,n,\omega)\to X_{S'}(k,n,\omega). \]
First we prove that it is surjective on the level of fixed points.
\begin{prop}\label{prop:surjective-on-fixed-points}
    For $\omega \in \bN$, $k \in [n]$ and $S,S' \subset [n]$ with $S' \subset S$,
\[{\rm pr}_{S \setminus S'} \, {\mathcal Jug}_S(k,n,\omega) = {\mathcal Jug}_{S'}(k,n,\omega). \]
\end{prop}
\begin{proof}
By construction the image of ${\mathcal Jug}_S(k,n,\omega)$ is contained in ${\mathcal Jug}_{S'}(k,n,\omega)$ provided $S' \subset S$. Now we explicitly construct one preimage $\mathcal{J}$ for each $\mathcal{J}' \in {\mathcal Jug}_{S'}(k,n,\omega)$. Let $s_0 \in S$ such that $s_0 \notin S'$ with $s'_t < s_0 <s'_{t+1}$ and $[s'_{t}+1,s_0-1] \cap S = \varnothing$. Below we use $S$ instead of $\bZ_r$ as index set of the juggling pattern $\mathcal{J} = (J_s)_{s\in S}$ to simplify the notation.
For $s \in S \cap S'$ define $J_s:=J_s'$. In position $s_0$ we define
\[ J_{s_0} := \tau_{s_0-s'_t}J'_{{t}} \]
If $\# \tau_{s_0-s'_t}J'_{{t}}  = \omega k$, this is sufficient. Otherwise add the element 
\[ p_0 := \mr{max}\{ p \in [\omega n] \setminus J_{s_0} : p \in ( \tau_{s_0-s'_{t+1}}J'_{{t+1}} \cup [\omega n+s_0-s'_{t+1}+1,\omega n])\}\]
until $J_{s_0}$ has the right cardinality.
Continue the above procedure until there exists no $s_0$ with the above property. By construction, the resulting $\mathcal{J}$ is contained in ${\mathcal Jug}_S(k,n,\omega)$. Clearly, its image under ${\rm pr}_{S \setminus S'}$ is $\mathcal{J'}$.
\end{proof}
\begin{prop}\label{prop:sujection-on-cells}
    For $\omega \in \bN$, $k \in [n]$ and $S,S' \subset [n]$ with $S' \subset S$,
every cell in $ X_{S'}(k,n,\omega)$ is the image of a cell in $X_S(k,n,\omega)$. Namely, for $\mathcal{J} \in {\mathcal Jug}_S(k,n,\omega)$, the attracting set of ${\rm pr}_{S \setminus S'}(p_\mathcal{J})$ in $ X_{S'}(k,n,\omega)$ equals to the image of the attracting set $W_\mathcal{J} \subset X_S(k,n,\omega)$ under the map ${\rm pr}_{S \setminus S'}$.
\end{prop}
\begin{proof}
From \cite[Theorem~4.13]{Pue2020} we obtain that the vector spaces of a point in the cell $W_\mathcal{J}$ in $X_S(k,n,\omega)$ are spanned by vectors of the form 
\begin{equation}\label{eqn:coordinates-in-cells}
    w_j^{(i)} := v_j^{(i)} + \sum_{\ell \in [j+1,\omega n] \setminus J_i} \mu_{j,\ell}^{(i)} v_\ell^{(i)} \quad \mr{for} \ \mr{all} \ i \in \Delta_r, \ j\in J_i
\end{equation}
subject to the equations $\mu_{j+q_i,\ell+q_i}^{(i+1)} = \mu_{j,\ell}^{(i)}$ if $j +q_i \leq \omega n$ and $\ell+q_i \notin J_{i+1}$.
Let $W_\mathcal{J'}$ be a cell in $ X_{S'}(k,n,\omega)$. Now we consider the extension  to a cell $W_\mathcal{J}$ in $X_S(k,n,\omega)$ such that ${\rm pr}_{S \setminus S'} W_\mathcal{J}= W_\mathcal{J'}$.  As for the fixed points take $s_0 \in S$ such that $s_0 \notin S'$ with $s'_t < s_0 <s'_{t+1}$ and $[s'_{t}+1,s_0-1] \cap S = \varnothing$. Again we use $S$ instead of $\bZ_r$ as index for the vertices.
For $s \in S \cap S'$ take the basis vectors $w^{(s)}_j$ with $j \in J_s$ as defined above. Here $\mathcal{J} = (J_s)_{s \in S}$ is obtained from $\mathcal{J}' \in {\mathcal Jug}_{S'}(k,n,\omega)$ as described in Proposition~\ref{prop:surjective-on-fixed-points}. In position $s_0$ we define
\[ w^{(s_0)}_j := \mathcal{T}_{q}w^{(t)}_{j-q}  \]
for $q:=s_0-s'_t$ and $j \in \tau_{q}J'_{{t}}$. For $\# \tau_{q}J'_{s_{t}} = \omega k$ this completes the basis. Otherwise we extend it by the elements
\[ w^{(s_0)}_p := \begin{cases} \mathcal{T}_{s_0-s'_{t+1}}w^{({t+1})}_{p+s'_{t+1}-s_0} & {\rm if} \ p \in \tau_{s_0-s'_{t+1}} J_{t+1},\\ v^{(s_0)}_{p} & {\rm if} \ p \in[\omega n +s_0-s'_{t+1} +1,\omega n] \end{cases}   \]
where $p \in J_i$ is obtained as in the proof of Proposition~\ref{prop:surjective-on-fixed-points}. We repeat this procedure for all remaining $s_0 \in {S \setminus S'}$. 
By construction this is compatible with the describing equations of the cells from \cite[Theorem~4.13]{Pue2020}. Moreover this subset of $W_\mathcal{J}$ is mapped surjectively to $W_\mathcal{J'} \subseteq X_{S'}(k,n,\omega)$. In particular, this is the smallest cell from $ X_{S}(k,n,\omega)$ which surjects to $W_\mathcal{J'}$. For the other cells whose fixed points mapping to the same fixed point, surjectivity is proven in the same way. The $\bC^*$-action commutes with the projection ${\rm pr}_{S \setminus S'}$. Hence the image of any attracting set of a fixed point $p$ has to be contained in the attracting set of ${\rm pr}_{S \setminus S'}(p)$.
\end{proof}
\begin{dfn}\label{dfn:energy-juggling-pattern}
 Let $\omega \in \bN$, $k \in [n]$ and $S \subset [n]$ with $\# S=r$.  The energy of a juggling pattern $\mathcal{J} \in {\mathcal Jug}_S(k,n,\omega)$ is 
 \[ e(\mathcal{J}):=\sum_{i \in \bZ_r} \sum_{j \in J_i \setminus \tau_{q_{i-1}}J_{i-1}} \# [j+1,\omega n] \setminus J_i.\]
\end{dfn}
\begin{cor}\label{cor:cell-dim-via-jugg}
 Let $\omega \in \bN$, $k \in [n]$ and $S \subset [n]$ with $\# S=r$. Then $\dim_\bC W_{\mathcal{J}} = e(\mathcal{J})$ for all $\mathcal{J} \in {\mathcal Jug}_S(k,n,\omega)$.
\end{cor}
\begin{proof}
    This follows immediately from the explicit coordinate description of the cell $W_\mathcal{J} \subset X_S(k,n,\omega)$ in the proof of Proposition~\ref{prop:sujection-on-cells}.
\end{proof}
\begin{rem}
For $S=[n]$ and $\omega = 1$, the elements of ${\mathcal Jug}_S(k,n,\omega)$ describe the different possibilities to juggle with $k$-many balls. Definition~\ref{dfn:energy-juggling-pattern} is linked to the energy which is required by the juggling person to throw this pattern.
\end{rem}
\begin{proof}[Proof of Theorem~\ref{thm:natural-projections}]
The statement follows from Proposition~\ref{prop:surjective-on-fixed-points} in combination with ~\ref{prop:sujection-on-cells}.
\end{proof}

\subsection{Geometric properties}\label{sec:more-geometry}
\begin{prop}\label{prop:dim-preserving-projection}
For $\omega\in \bN$, $k \in [n]$, the projection $\mr{pr}_i$ preserves the dimension of a top-dimensional cell $W_\mathcal{J} \subset X(k,n,\omega)$ with $\mathcal{J} \in {\mathcal Jug}(k,n,\omega)$ if and only if $1 \notin J_i$ or $1,\omega n \in J_i$.
\end{prop}
\begin{proof}
By \cite[Theorem~2.11]{FLP23}, the top-dimensional cells of $X(k,n,\omega)$ are parameterized by the $k$-element subsets $I$ of $[n]$. The associated juggling pattern is uniquely determined by the condition $1 \in J_i$ for all $i \in I$.
Recall the explicit coordinates for the attracting sets from Equation~\eqref{eqn:coordinates-in-cells}:
\[     w_j^{(i)} := v_j^{(i)} + \sum_{\ell \in [j+1,\omega n] \setminus J_i} \mu_{j,\ell}^{(i)} v_\ell^{(i)} \quad \mr{for} \ \mr{all} \ i \in \Delta_n, \ j\in J_i \]
subject to the equations  $\mu_{j+1,\ell+1}^{(i+1)} = \mu_{j,\ell}^{(i)}$ if $j +1 \leq \omega n$ and $\ell+1 \notin J_{i+1}$. By assumption $W_\mathcal{J} \subset X(k,n,\omega)$ is top-dimensional. This implies that $1\in J_i$ for some $i \in \bZ_n$. If $\omega n \notin J_i$, this implies that $\# [\omega n] \setminus J_i > \#[2,\omega n] \setminus J_{i+1}$. Hence $\mr{pr}_{i_0}$ preserves the dimension of $W_\mathcal{J}$ if and only if $1 \notin J_{i_0}$ or $1,\omega n \in J_{i_0}$, because this covers all cases where the number of independent coefficients $ \mu_{j,\ell}^{(i)}$ is preserved by the projection $\mr{pr}_{i_0}$.
\end{proof}
\begin{prop}
Let $\omega \in \bN$, $k \in [n]$ and $T \subset [n]$ with $\# T=q$. For any  permutation $\sigma \in \Sigma_q$   
\[ {\rm pr}_{t_1} \circ \hdots \circ {\rm pr}_{t_q}(U) = {\rm pr}_{t_{\sigma(1)}} \circ \hdots \circ {\rm pr}_{t_{\sigma(q)}}(U) \]
for all $U \in X(k,n,\omega)$. 
\end{prop}
\begin{proof}
We use the cellular decomposition into attracting sets of fixed points as described in Lemma~\ref{lem:cell-decomp}. Then every point has coordinates
\[ U = \Big( U^{(i)} = \mr{span}\big\{ w_j^{(i)} : j \in J_i \big\}\Big)_{i \in \bZ_n}\]
for some $\mathcal{J} \in {\mathcal Jug}(k,n,\omega)$. The projection $\mr{pr}_t$ sends $U$ to 
\[ U = \Big( U^{(i)} = \mr{span}\big\{ w_j^{(i)} : j \in J_i \big\}\Big)_{i \in \bZ_n\setminus\{t\}}.\]
By Theorem~\ref{thm:natural-projections} we know that $\mr{Im} \, \mr{pr}_t = X_{[n]\setminus\{t\}}(k,n,\omega)$. In particular, $\mr{pr}_t(U) \in W_{\mr{pr}_t\mathcal{J}}$. On the cells of $X_{[n]\setminus\{t\}}(k,n,\omega)$, the projections are defined in the same way as above. Hence $\mr{pr}_i \circ \mr{pr}_j U = \mr{pr}_j \circ \mr{pr}_i U$ for all $i,j\in \bZ_n$ and the claim follows by induction.
\end{proof}
\begin{lem}\label{lem:parametrization-of-irreducible-components}
 Let $\omega \in \bN$, $k \in [n]$ and $S \subset [n]$ with $\# S=r$. The irreducible components of $X_S(k,n,\omega)$ are parameterized by the set
 \[ \mathcal{I}rr_S(k,n) := \Big\{ I \in \binom{[n]}{k} : i \in I \ \mr{and} \ i \notin S \ \mr{implies} \ i+1 \in I \Big\}. \]
The closed stratum $\overline{\mathcal{S}_{U_I}}$ of the $\Delta_r$-representation
\[ U_I := \mr{pr}_{[n]\setminus S} \Big( \bigoplus_{i \in I}U_i(\omega n) \Big) \quad \mr{for} \  I \in \mathcal{I}rr_S(k,n)\]
is an irreducible component and all irreducible components are of this form.
\end{lem}

The proof of this statement is postponed until Section~\ref{sec:Poset-structures} because it requires a switch between the different combinatorial models for the cell structure of $X_S(k,n,\omega)$. Observe that the set $\mathcal{I}rr_S(k,n)$ does not depend on $\omega$.
\begin{cor}\label{cor:equidimensional-irreducible-components}
The irreducible components of $X_S(k,n,\omega)$ are equidimensional.
\end{cor}
\begin{proof}
By Proposition~\ref{prop:dim-preserving-projection}, $\mathcal{I}rr_S(k,n)$ from Lemma~\ref{lem:parametrization-of-irreducible-components} contains exactly the index sets parametrizing the top-dimensional cells of $X(k,n,\omega)$ whose dimension is preserved by the projection $\mr{pr}_{[n]\setminus S}$. Lemma~\ref{lem:parametrization-of-irreducible-components} also implies that there are no other irreducible components.
\end{proof}
\begin{thm}\label{thm:geometric-properties}
For $\omega\in \bN$, $k \in [n]$ and $S \subset [n]$ the variety $X_S(k,n,\omega)$ satisfies the following:
\begin{enumerate}[label=(\roman*)]
    \item it is a projective variety of dimension $\omega k(n-k)$;
    \item the BB-decomposition is a cellular decomposition;    
    \item its irreducible components are equidimensional;
    \item the irreducible components are normal, Cohen-Macaulay and have rational singularities;
    \item the irreducible components are parameterized by the set
 \[ \mathcal{I}rr_S(k,n) = \Big\{ I \in \binom{[n]}{k} : i \in I \ \mr{and} \ i \notin S \ \mr{implies} \ i+1 \in I \Big\}. \]
\end{enumerate}
\end{thm}
\begin{proof}
From Theorem~\ref{thm:natural-projections} it follows that $X(k,n,\omega)$ surjects on $X_S(k,n,\omega)$ which surjects on $X_{\{i\} }(k,n,\omega)$. By \cite[Theorem~2.11]{FLP23} we obtain that $X(k,n,\omega)$ is of dimension $\omega k(n-k)$. $X_{\{i\}}(k,n,\omega)$ is a quiver Grassmannian for the loop quiver and in \cite[Lemma~4.9]{Pue2020} its dimension is computed as $\omega k(n-k)$. Hence $X_S(k,n,\omega)$ is of the same dimension. Part (ii) is proven in Lemma~\ref{lem:cell-decomp}. The third part is proven in Corollary~\ref{cor:equidimensional-irreducible-components}. \cite[Lemma~4.12]{Pue2020} implies item (iv). The last part is shown in Lemma~\ref{lem:parametrization-of-irreducible-components}.  
\end{proof}

\section{Moment graph and poset structure}\label{sec:moment-graph+Cohomology}
\subsection{Moment Graph}\label{sec:moment-graph}
In this section we provide a combinatorial description for the moment graph of $X_S(k,n,\omega)$. More detail can be found in \cite{GKM1998,LaPu2020,LaPu2021}.
This graph captures the structure of one-dimensional orbits between fixed points for suitable torus actions on complex projective varieties. This helps to understand the equivariant geometry of the variety. 

Let a torus $T$ act on a complex projective algebraic variety $X$ with finitely many fixed points and finitely many one-dimensional $T$ orbits (i.e. the action is skeletal). The definition below is specialized to our settings.
\begin{dfn}Let $T$ be an algebraic torus and let $X$ be a complex projective algebraic $T$-variety. Assume that $X$ admits a skeletal $T$-action and a $T$-stable cellular decomposition where every cell has exactly one fixed point. The corresponding moment graph $\mathcal{G}(X,T)$ is given by 
\begin{itemize}
    \item the vertex set is the fixed point set: $\mathcal{V}=X^T$;
    \item there is an edge $x\to y$ if and only if $x$ and $y$ belong to the same one dimensional $T$ orbit closure $\overline{\mathcal{O}_{x\to y}}$ and $y$ belongs to the closure of the cell containing $x$;
    \item the label of the edge $x\to y$ is the character $\alpha\in {\rm Hom}(T,\bC^*)$ for the torus action on $\mathcal{O}_{x\to y}$. 
\end{itemize}
\end{dfn}

The edge labels are well defined up to a sign, but since this does not play any role in the applications (e.g. computation of equivariant cohomology), we assume the labels to be fixed once and for all, and forget about this ambiguity.

\begin{prop}\label{prop:bijection_jug-cells}
For $\omega\in \bN$, $k \in [n]$ and $S \subset [n]$ with $\#S =r$, there is a bijection between ${\mathcal Jug}_S(k,n,\omega)$ and the set ${\mathcal C}_S(k,n,\omega)$ defined as
\[ \Big\{ (\ell_j)_{j \in \bZ_n} \in [0,\omega r]^{\bZ_n} :  \bdim \,  \Big(\bigoplus_{j \in \bZ_n} U({\underline{j}};\ell_j) \Big) = (k\omega, \dots, k\omega) \in \bN^{\bZ_r} \Big\}, \]
with 
$\underline{j} := \max\{ i \in \bZ_r : s_i \leq j < s_{i+1} \ \ {\mod n} \}$.
\end{prop}

\begin{proof}
From \cite[Proposition~3.2]{FLP23} we obtain that the $T$-fixed points of $X(k,n,\omega)$ are parametrized by 
    \[ \Big\{ (\ell_j)_{j \in \bZ_n} \in [0,\omega n]^{\bZ_n} :  \bdim \,  \Big(\bigoplus_{j \in \bZ_n} U(j;\ell_j) \Big) = (k\omega, \dots, k\omega) \in \bN^{\bZ_n} \Big\}. \]
    By Theorem~\ref{thm:natural-projections} we know that $\mr{pr}_{[n] \setminus S}$ surjects the representations on the right hand side of this set to the fixed points of $X_S(k,n,\omega)$. The projection $\mr{pr}_{[n] \setminus S}$ sends the indecomposable $\Delta_n$-representation $U(j;\omega n)$ ending over the vertex $j \in \bZ_n$ to the indecomposable $\Delta_r$-representation $U(\underline{j};\omega r)$ ending over $\underline{j} \in \bZ_r$. Hence the image of $\mr{pr}_{[n] \setminus S}$ is described by ${\mathcal C}_S(k,n,\omega)$.
\end{proof}
Let $\omega\in \bN$, $k \in [n]$ and $S \subset [n]$ with $\#S =r$. Now we describe certain cut and paste moves on the segments of the elements in $\mathcal{C}_S({k,n,\omega})$.
For every element $\ell_\bullet \in \mathcal{C}_S({k,n,\omega})$ there are maps of the form $f_{i,j,q} : \mathcal{C}_S({k,n,\omega}) \to \mathcal{C}_S({k,n,\omega})$ with
\[  \Big(f_{i,j,q}\big(\ell_\bullet\big)\Big)_s := \begin{cases} \ell_s & s \notin \{i,j\}\\ \ell_i-q & s=i\\  \ell_j+q & s=j\end{cases}.\]
whenever $q \in [0,\min\{\ell_i,\omega r-\ell_j\}]$ and $\underline{i}-\ell_{i} = \underline{j} - \ell_j-q \mod r$. Here $\underline{i} \in \bZ_r$ for $i \in \bZ_n$ is defined as in Proposition~\ref{prop:bijection_jug-cells}. It is straightforward to check that $f_{i,j,q}\big(\ell_\bullet\big)$ is again an element of $\mathcal{C}_S({k,n,\omega})$. These cut and paste moves describe all one-dimensional $T$-orbits.
\begin{lem}\label{lem:moment-graph}
    The vertices of the moment graph for the action of the torus $T$ on $X_S(k,n,\omega)$ are labelled by the elements of ${\mathcal C}_S({k,n,\omega})$. There is an oriented edge in the moment graph from $\ell_\bullet$ to $f_{i,j,q}\big(\ell_\bullet\big)$ if and only if $\ell_i > \ell_j+q$. The label of the edge $\ell_\bullet \to f_{i,j,q}\big(\ell_\bullet\big)$ is $\epsilon_j-\epsilon_i +\delta\cdot(\ell_i-\ell_j-q)$, where $\delta(\gamma):=\gamma_0$ and $\epsilon_i(\gamma):=\gamma_i$ for any $i \in [n]$ and $\gamma=(\gamma_0,\gamma_1,\dots,\gamma_n) \in T$.
\end{lem}
\begin{proof}
The one-dimensional torus orbits in quiver Grassmannians for nilpotent $\Delta_n$-representations are described in \cite[Theorem~6.15]{LaPu2020}. This applies here since $U_{\omega,S}$ is nilpotent. The explicit description of the edges and their labels follows from the translation of \cite[Theorem~6.15]{LaPu2020} to the combinatorial language of Proposition~\ref{prop:bijection_jug-cells}.
\end{proof}

\subsection{T-equivariant cohomology}\label{sec:T-equi-cohomology}
The description of the moment graph $\mathcal{G}$ from Lemma~\ref{lem:moment-graph} allows to compute the ($T$-equivariant) cohomology ring of $X:=X_S(k,n,\omega)$ (cf.  \cite[Theorem~6.6]{LaPu2020}). Consider $R:=\mathbb{Q}[\epsilon_1, \ldots, \epsilon_n, \delta]$ as a $\mathbb{Z}$-graded ring for the grading $\deg(\epsilon_i)=\deg(\delta)=2$ for all $i\in [n]$. Let $\alpha(\ell_\bullet,\ell'_\bullet)$ denote the label of the edge $\ell_\bullet \to \ell'_\bullet$. 

The translation of \cite[Theorem~1.2.2]{GKM1998} to the combinatorial description of the moment graph from Lemma~\ref{lem:moment-graph} yields the following result.
\begin{cor}\label{cor:GKM}
There is an isomorphism of ($\mathbb{Z}$-graded) rings
\[
H_T^\bullet\big(X,\mathbb{Q}\big)\simeq \left\{\big(z_{\mathcal{\ell_\bullet}}\big)_{\ell_\bullet \in\mathcal{C}_S({k,n,\omega})} \in\bigoplus_{\ell_\bullet \in\mathcal{C}_S({k,n,\omega})}R \, \Bigg\vert   \begin{array}{c}
  z_{\ell_\bullet}\equiv z_{\ell_\bullet'}\mod \alpha(\ell_\bullet,\ell'_\bullet)\\
\mathrm{for} \ \mathrm{every} \ \mathrm{edge} \ \ell_\bullet\rightarrow \ell_\bullet' \end{array}\right\}.
\]
\end{cor}
\begin{rem}
$H_T^\bullet(X_S(k,n,\omega),\bQ)$ admits a very nice basis as a free module over $R$, called Knutson-Tao (KT) basis (see \cite[Definition~3.2, Theorem~3.22]{LaPu2021}). 
\end{rem}

\subsection{Poset structures on the set of fixed points}\label{sec:Poset-structures}

In this subsection we identify the different combinatorial models describing the cell structure of $X_S(k,n,\omega)$.
For $\omega\in \bN$, $k \in [n]$ and $S \subset [n]$ with $\#S =r$, we introduce partial orders on the sets $X_S(k,n,\omega)^T$, ${\mathcal Jug}_S(k,n,\omega)$ and ${\mathcal C}_S(k,n,\omega)$ and study their relation under the bijections from Lemma~\ref{lem:T-fixed-point-parametrisation} and Proposition~\ref{prop:bijection_jug-cells}.

For $p,p' \in X_S(k,n,\omega)^T$ we write $p' \preceq p$ if $\overline{ C_p}$ contains $p'$. By Theorem~\ref{trm:cells-are-strata} we obtain the same partial order $\preceq$ if we consider closures of $\mr{Aut}_{\Delta_r}(U_{\omega r})$-orbits. Given $\mathcal{J}_\bullet, \mathcal{J}_\bullet' \in {\mathcal Jug}_S(k,n,\omega)$ we write $\mathcal{J}_\bullet \geq_{\mathcal J} \mathcal{J}_\bullet'$ iff $j^{(i)}_q \leq j'^{(i)}_q$ for all $i \in \bZ_r$ and $q \in[k\omega]$ where we order each $J_i \in \binom{[n\omega]}{k\omega}$ as
\[ \big(j^{(i)}_1 <  j^{(i)}_2 < \ldots < j^{(i)}_{k\omega} \big). \]
For two elements $\ell_\bullet, \ell_\bullet' \in {\mathcal C}_S(k,n,\omega)$ we write $\ell_\bullet \geq_{\mathcal C} \ell_\bullet'$ if there exists an oriented path from $\ell_\bullet$ to $\ell_\bullet'$ in the moment graph for the $T$-action on $X_S(k,n,\omega)$ (as described in Lemma~\ref{lem:moment-graph}). 

\begin{thm}\label{thm:poset-iso}
For $\omega\in \bN$, $k \in [n]$ and $S \subset [n]$ with $\#S =r$, there are order preserving poset isomorphisms between  $X_S(k,n,\omega)^T$, ${\mathcal Jug}_S(k,n,\omega)$ and ${\mathcal C}_S(k,n,\omega)$. 
\end{thm}
\begin{proof}
The isomorphisms on the level of sets were introduced in Lemma~\ref{lem:T-fixed-point-parametrisation} and Proposition~\ref{prop:bijection_jug-cells}. These isomorphisms commute with the projections $\mr{pr}_{S\setminus S'}$ and the poset structures as defined above are induced by the poset structures for the case $S=[n]$ as studied in \cite[Section~4.3]{FLP23}. Hence the claim follows from \cite[Theorem~4.6]{FLP23}.
\end{proof}
\begin{cor}\label{cor:about-cells-and-their-closures}
The closure of every cell in $X_S(k,n,\omega)$ is obtained as
\[ \overline{ C_\mathcal{J}} = \bigcup_{\mathcal{J}' \in {\mathcal Jug}_S(k,n,\omega) \ \mathrm{s.t.:} \ \mathcal{J}' \leq_{\mathcal \mathcal{J}} \mathcal{J}} C_{\mathcal{J}'}.\]
Moreover the moment graph of $\overline{ C_\mathcal{J}}$ is the full subgraph of the graph described in Lemma~\ref{lem:moment-graph} on the vertices corresponding to $\mathcal{J}' \leq_{\mathcal J} \mathcal{J}$. The dimension of $C_\mathcal{J}$ is the number of edges in the moment graph starting at $\mathcal{J}$. This equals the energy of the juggling pattern $\mathcal{J}$.
\end{cor}
\begin{proof}
By Theorem~\ref{trm:cells-are-strata}, every cell closure in $X_S(k,n,\omega)$ is a union of $\mr{Aut}_{\Delta_r}(U_{\omega r})$-orbits. Now the poset isomorphism from Theorem~\ref{thm:poset-iso} implies the desired description of the closure.
In particular, the moment graph of any cell closure is the full subgraph on the vertices which are smaller with respect to any of the partial orders. The dimension formulas are obtained from \cite[Corollary~6.5]{LaPu2020} and Corollary~\ref{cor:cell-dim-via-jugg}.
\end{proof}
\begin{proof}[Proof of Lemma~\ref{lem:parametrization-of-irreducible-components}]
By definition of $\mathcal{I}rr_S(k,n)$ it follows that each element parametrizes an irreducible component of $X_S(k,n,\omega)$. It remains to show that there are no other irreducible components. 
This is equivalent to show that the image of $\mathcal{S}_{U_I} \subset X(k,n,\omega)$ under $\mr{pr}_{[n] \setminus S}$ for every $I \in \binom{[n]}{k} \setminus \mathcal{I}rr_S(k,n)$ is contained in $\overline{S_{U_{I'}}} \subset X_S(k,n,\omega)$ for some $I' \in \mathcal{I}rr_S(k,n)$. It is sufficient to prove this for the top-dimensional cells in each of the strata. Let $\mathcal{J} \in {\mathcal Jug}_S(k,n,\omega)$ represent the top-dimensional cell of $\mr{pr}_{[n] \setminus S} \mathcal{S}_{U_I}$ for $I \in \binom{[n]}{k} \setminus \mathcal{I}rr_S(k,n)$. By assumption there exists an $i \in I$ such that $i \notin S$ and $i+1 \notin I$. Let $\mathcal{J}'$ be the juggling pattern corresponding to $I':= I\cap\{i+1\} \setminus \{ i\}$. By Lemma~\ref{lem:moment-graph}, there exists a path in the moment graph from $\mathcal{J}'$ to $\mathcal{J}$. Hence, the closure of $W_{\mathcal{J}'}$ contains $W_{\mathcal{J}}$ by the poset isomorphism from Theorem~\ref{thm:poset-iso}. Now the claim follows because we can inductively apply this procedure until we arrive at $I' \in \mathcal{I}rr_S(k,n)$.
\end{proof}
\subsection{Poincar\'e polynomials}
For $\omega \in \bN$, $k \in [n]$ and $S \subset [n]$, the Poincar\'e polynomial of $X_S(k,n,\omega)$ is defined as 
\[ P_S^{k,n,\omega}(q) := \sum_{p \in X_S(k,n,\omega)^T} q^{\dim_\bC W_p}.\]
\begin{rem}
Each combinatorial model for the cells structure provides a different method to compute $P_S^{k,n,\omega}(q)$. Depending on the parameters it is convenient to switch
between the models.  
\end{rem}
\begin{lem}\label{lem:Poincare-Polynomial}
For $\omega \in \bN$, $k \in [n]$ and $S \subset [n]$ with $\# S = r$, the Poincar\'e polynomial of $X_S(k,n,\omega)$ is
\[ P_S^{k,n,\omega}(q) = \sum_{\mathcal{J} \in \mathcal{J}ug_S({k,n,\omega})} q^{e(\mathcal{J})}.\]
Here $e(\mathcal{J})$ denotes the energy of the juggling pattern as in Definition~\ref{dfn:energy-juggling-pattern}.
\end{lem}
\begin{proof}
In Lemma~\ref{lem:T-fixed-point-parametrisation}, the bijection between the cells of $X_S(k,n,\omega)$ and elements of $\mathcal{J}ug_S({k,n,\omega})$ is established. The dimension of a cell is computed from the corresponding juggling pattern as described in Corollary~\ref{cor:cell-dim-via-jugg}.
\end{proof}

\section{Desingularization}\label{sec:desingularization}
Based on the general construction from \cite{PuRe2022} we construct an explicit desingularization of $X_S(k,n,\omega)$ (see also \cite{CFR13,FF13,KS14,S17}).

The idea of the following construction is to associate each irreducible component of
$X_S(k,n,\omega)$ a representations of an extended quiver such that certain quiver Grassmannians for these representation are smooth and resolve the singularities of the irreducible components of $X_S(k,n,\omega)$.

Let $\hat{\Delta}_r$ be the quiver with vertex set
\[ \big\{ (i,j) \ : \ i \in \bZ_r \ \mr{and} \ j \in [\omega r] \big\} \]
and arrows
\begin{align*}
&\big\{ a_{i,j} : (i,j)\to (i,j+1) \ : \ i \in \bZ_r \ \mr{and} \ j \in [\omega r-1] \big\} \cup\\
&\big\{ b_{i,j} : (i,j)\to (i+1,j-1) \ : \ i \in \bZ_r \ \mr{and} \ j \in [\omega r]\setminus \{1\} \big\}.
\end{align*}
For $M \in \mr{rep}_\bC^{\omega r}({\Delta}_r)$ (see Remark \ref{nilp}) its image under $\Lambda: \mr{rep}_\bC({\Delta}_r) \to \mr{rep}_\bC(\hat{\Delta}_r)$ is defined as 
 \[ \hat{M} := \big( (\hat{M}^{(i,j)})_{i \in \bZ_r, j \in [\omega r] }, (\hat{M}_{a_{i,j}}, \hat{M}_{b_{i,j+1}})_{i \in \bZ_r, j \in [\omega r-1]} \big)\]
 with 
 \begin{align*}
 \hat{M}^{(i,1)} &:= M^{(i)}   &\mr{for} \ j =1\\
\hat{M}^{(i,j)} &:= M_{i+j-2} \circ  M_{i+j-3} \circ  \dots \circ M_{i+1} \circ  M_{i} (M^{(i)})  &\mr{for} \ j \geq 2\\
\hat{M}_{a_{i,j}} &:= M_{i+j-1}  &\mr{for} \ j \geq 1\\
\hat{M}_{b_{i,j}} &:= \iota : \hat{M}^{(i,j)}  \hookrightarrow \hat{M}^{(i+1,j-1)}  &\mr{for} \ j \geq 2
 \end{align*}
Here the inclusion maps along $b_{i,j}$ arise naturally from the definition of the vector spaces of $\hat{M}$.

There is a corresponding restriction functor $\mr{res} : \mr{rep}_\bC(\hat{\Delta}_r) \to \mr{rep}_\bC({\Delta}_r)$ such that the image of $W \in \mr{rep}_\bC(\hat{\Delta}_r)$ is
\[ \mr{res}  \, W := \Big( \big(W^{(i,1)}\big)_{i \in \bZ_r}, \big(W_{b_{i,2}} \circ W_{a_{i,1}}\big)_{i \in \bZ_r} \Big).\]

\begin{dfn}
 Let $\omega\in \bN$, $k \in [n]$ and $S \subset [n]$. For $I \in \mathcal{I}rr_S(k,n)$ set $\hat{X}^I_S(k,n,\omega):= \mr{Gr}_{\bdim \, \hat{U}_I} \big(\hat{U}_{\omega n,S} \big)$ where $U_I$ is defined as in Lemma~\ref{lem:parametrization-of-irreducible-components}. The map
\[ \pi_I : \hat{X}^I_S(k,n,\omega) \longrightarrow X_S(k,n,\omega) \]
is defined by $\pi_I(V) := \mr{res} V$ for all $V \in \hat{X}^I_S(k,n,\omega)$.   
\end{dfn}
\begin{rem}
The vector spaces of $\hat{U}_{\omega n}$ are spanned by subsets of the bases for the vector spaces of $U_{\omega n}$. Hence the $T$ action on $U_{\omega n}$ extends to the quiver Grassmannians $\hat{X}_S^I(k,n,\omega)$ for $I \in \binom{[n]}{k}$ in the obvious way. The same holds for the $\bC^*$ action.
\end{rem}
The following result is a special case of \cite[Theorem~3.18, Lemma~5.3]{PuRe2022}. 
\begin{thm}\label{thm:desing}
The map
\[ \pi := \bigsqcup_{I \in \mathcal{I}rr_S(k,n)} \pi_I  \ : \ \bigsqcup_{I \in \mathcal{I}rr_S(k,n)} \hat{X}_S^I(k,n,\omega) \ \longrightarrow \  X_S(k,n,\omega)\]
 is a $T$-equivariant desingularization of $X_S(k,n,\omega)$.
\end{thm}
\begin{rem}
 The quiver Grassmannians $\hat{X}_S^I(k,n,\omega)$ are irreducible by \cite[Proposition 37]{CEFR2018}.
\end{rem}
The next result is a special case of \cite[Theorem~3.21]{PuRe2022} and generalizes \cite[Theorem~7.10]{FLP21}.
\begin{thm}\label{trm:tower-of-grassmann-bundles}
For each $I \in \mathcal{I}rr_S(k,n)$ the quiver Grassmannian $\hat{X}_S^I(k,n,\omega)$ is isomorphic to a tower of fibrations 
\[  \hat{X}_S^I(k,n,\omega) = X_1 \to X_{2} \to \dots \to X_{\omega r } = \mr{pt} \]
where each map $X_j \to X_{j+1}$ for $j \in [\omega r-1]$ is a fibration with fiber isomorphic to a product of Grassmannians of subspaces.
\end{thm} 
\subsection{Natural projections}\label{sec:natural-proj-desing}
As in Section~\ref{sec:natural-projections}, the maps
\[ {\rm pr}_j: \prod_{i \in \bZ_n} {\rm Gr}_{\omega k}(\bC^{\omega n}) \to \prod_{i \in \bZ_n \setminus \{j\} }{\rm Gr}_{\omega k}(\bC^{\omega n})  \quad {\rm for} \ j \in \bZ_n \]
naturally induce projections between the desingularizations. 
\begin{rem}
Let $S \subset [n]$, $S=(s_1 < s_2 < \hdots < s_r)$. For $m \in [\omega n]$ let $\overline{m} \in [\omega r]$ be the image of $m$ under the induced projection $\mr{pr}_S$, keeping in mind the identification of $S\subset [n]$ with $[r]$.
There is 
an induced functor $\mr{pr}_{[n]\setminus S} : \mr{rep}_\bC^{\omega n} \hat{\Delta}_n \to \mr{rep}_\bC^{\omega r} \hat{\Delta}_r$ where an object $U \in \mr{rep}_\bC^{\omega n} \hat{\Delta}_n$ is sent to $V \in \mr{rep}_\bC^{\omega r} \hat{\Delta}_r$ with
\begin{align*}
    V^{(i,\overline{m})}&:= U^{(s_i,m)} \quad \mr{such} \ \mr{that} \  (i+m-1\mod n)\notin [n]\setminus S,\\
    V_{\alpha_{i,\overline{m}}}&:= U_{\alpha_{s_i,m+s_{i+1}-s_i-1}}\circ \dots \circ  U_{\alpha_{s_i,m}},\\
    V_{\beta_{i,\overline{m}}}&:= U_{\beta_{s_{i+1}-1,m-s_{i+1}+s_i-1}}\circ \dots \circ  U_{\beta_{s_i,m}}.
\end{align*}
\end{rem}

\begin{dfn}
Let $\hat{X}_S(k,n,\omega)=\bigsqcup_{I \in \mathcal{I}rr_S(k,n)} \hat{X}_S^I(k,n,\omega).$
\end{dfn}
\begin{thm}\label{thm:natural-projections-desing}
For $\omega \in \bN$, $k \in [n]$ and $S,S' \subset [n]$ with $S' \subset S$, the map
\({\rm pr}_{S \setminus S'} : \hat{X}_{S}(k,n,\omega)\to \hat{X}_{S'}(k,n,\omega) \) is surjective.
\end{thm}
\begin{proof}
By Lemma~\ref{lem:parametrization-of-irreducible-components}, $\mathcal{I}rr_{S'}(k,n)$ parametrize the irreducible components of $X_{S'}(k,n,\omega)$. In Theorem~\ref{thm:natural-projections} it is shown that $X_{S}(k,n,\omega)$ surjects to $X_{S'}(k,n,\omega)$. Hence there is also a surjection of $\mathcal{I}rr_{S}(k,n)$ onto $\mathcal{I}rr_{S'}(k,n)$. 

Let $I \in \mathcal{I}rr_{S}(k,n)$ such that $\mr{pr}_{S\setminus S'}I=I'$ for a given $I' \in \mathcal{I}rr_{S'}(k,n)$. The construction of the desingularizations and the definition of the projections imply that the following diagram is commutative.
\begin{center}
\begin{tikzpicture}[scale=1]
\node at (-2,1) {$\hat{X}^I_S(k,n,\omega)$};
\node at (2,1) {$\hat{X}^{I'}_{S'}(k,n,\omega)$};

\node at (-2,-1) {${X}^I_S(k,n,\omega)$};
\node at (2,-1) {${X}^{I'}_{S'}(k,n,\omega)$};

\node at (0,1.3) {$\mr{pr}_{S \setminus S'}$};
\node at (0,-1.3) {$\mr{pr}_{S \setminus S'}$};

\node at (-2.5,0) {$\pi_I$};
\node at (2.5,0) {$\pi_{I'}$};

\draw[arrows={-angle 90}, shorten >=30, shorten <=30]  (-2,1) -- (2,1);

\draw[arrows={-angle 90}, shorten >=30, shorten <=30]  (-2,-1) -- (2,-1);

\draw[arrows={-angle 90}, shorten >=9, shorten <=9]  (-2,1) -- (-2,-1);

\draw[arrows={-angle 90}, shorten >=9, shorten <=9]  (2,1) -- (2,-1);

\end{tikzpicture}
\end{center}
Let $U \subset {X}^{I'}_{S'}(k,n,\omega)$ be open. By construction of the desingularization it follows that $\overline{\pi_{I'}^{-1}(U)}= \hat{X}^{I'}_{S'}(k,n,\omega)$.  Now the sujectivity of $\mr{pr}_{S \setminus S'}$ in the bottom row as proven in Theorem~\ref{thm:natural-projections} and the commutativity of the above diagram imply that 
\[ \pi_{I'}^{-1}(U) \subseteq{\rm pr}_{S \setminus S'} \hat{X}_{S}(k,n,\omega). \]
Hence the map in the top row of the diagram is surjective and the claim follows by application of the above argument to all irreducible components of ${X}^{I'}_{S'}(k,n,\omega)$.
\end{proof}
\subsection{Properties of the desingularization}
\begin{lem}(c.f. \cite[Theorem~5.5]{PuRe2022})\label{lem:cell-decomp-desing}
For every $I \in \mathcal{I}rr_S(k,n)$ the $T$-fixed points of $\hat{X}_S^I(k,n,\omega)$ are exactly the preimages of the $T$-fixed points of $X^I_S(k,n,\omega)$ under $\pi_I$ (where $X^I_S(k,n,\omega):=\overline{\mathcal{S}_{U_I}} \subset X_S(k,n,\omega)$). The $\bC^*$-attracting sets of these points provide a cellular decomposition of $\hat{X}^I_S(k,n,\omega)$.
\end{lem}
Recall (see \cite[Proposition~2.21]{FLP23})
that 
the automorphism group of $ \hat{U}_{\omega n}$ satisfies
\[  {\rm Aut}_{\hat{\Delta}_n}\big( \hat{U}_{\omega n}\big) \cong {\rm Aut}_{\Delta_n}\big(U_{\omega n}\big). \]

With the same argument we obtain the next result.
\begin{prop}\label{prop:aut-group-desing-partial-setting}
For $\omega\in \bN$, $k \in [n]$ and $S \subset [n]$ with $\#S=r$, the automorphism group of $ \hat{U}_{\omega n,S}$ satisfies
\[  {\rm Aut}_{\hat{\Delta}_r}\big( \hat{U}_{\omega n,S}\big) \cong {\rm Aut}_{\Delta_r}\big(U_{\omega n,S}\big). \]
\end{prop}
\begin{proof}
The relations on the matrices $A^{(i,1)}$ of $A \in {\rm Aut}_{\hat{\Delta}_r}( \hat{U}_{\omega n,S})$ from the compositions $b_{i+1,2} \circ a_{i,1}$ for all $i \in \mathbb{Z}_r$ are the same as for the matrices $B^{(i)}$ of $B \in {\rm Aut}_{\Delta_r}(U_{\omega n,S})$ (see Proposition~\ref{Aut}). All remaining components 
$A^{(i,r)}$ are obtained from the lower diagonal blocks of the matrices $A^{(i,1)}$ because of the explicit construction of the representation $\hat{U}_{\omega n,S}$. Hence all relevant information is already contained in the matrices $A^{(i,1)}$ which implies the desired isomorphism. 
\end{proof} 

\begin{prop}\label{lem:orbit-structures-desing}
The ${\rm Aut}_{\hat{\Delta}_r}\big( \hat{U}_{\omega n,S}\big) $-orbits of the $T$-fixed points in the quiver Grassmannian $\hat{X}^I_S(k,n,\omega)$ are their strata. The orbits of the $T$-fixed points under the induced ${\rm Aut}_{\hat{\Delta}_n}\big( \hat{U}_{\omega n}\big) $-action coincide with their $\bC^*$-attracting sets.
\end{prop}

\begin{proof}
The representation $\hat{U}_{\omega n,S}$ is an injective bounded $\hat{\Delta}_r$-representation. Hence we can apply \cite[Lemma~2.28]{Pue2019}, to conclude that all strata are ${\rm Aut}_{\hat{\Delta}_r}\big( \hat{U}_{\omega n,S}\big)$-orbits and vice versa. 

By construction the induced ${\rm Aut}_{\hat{\Delta}_n}\big( \hat{U}_{\omega n}\big) $-action commutes with the projection $\mr{pr}_{S\setminus S'}$. In \cite[Lemma~2.22]{FLP23} it is shown that the ${\rm Aut}_{\hat{\Delta}_n}\big( \hat{U}_{\omega n}\big) $-orbits in $X(k,n,\omega)$ coincide with the cells in the decomposition into the attracting sets of $\bC^*$-fixed points. Now Theorem~\ref{thm:natural-projections-desing} implies that $\mr{pr}_{S\setminus S'}$ surjects these cells to the attracting sets of $\bC^*$-fixed points in $X_S(k,n,\omega)$. Hence these sets have to coincide with the orbits for the induced ${\rm Aut}_{\hat{\Delta}_n}\big( \hat{U}_{\omega n}\big) $-action.
\end{proof} 
There exist strata containing more than one $T$-fixed point if $\#S < n$.



\begin{thebibliography}{99}

\bibitem[BB73]{Birula1973} A.~Bialynicki-Birula, \emph{Some theorems on actions of algebraic groups}, Annals of Mathematics, Second Series, \textbf{98} (1973), no. 3, 480--497.



\bibitem[CR08]{CaRe2008} P.~Caldero and M.~Reineke, \emph{On the quiver Grassmannian in the acyclic case}, J. Pure Appl. Algebra, Volume~\textbf{212} (2008), no. 11, 2369--2380.


\bibitem[CI11]{Cerulli2011} G.~Cerulli Irelli, \emph{Quiver Grassmannians associated with string modules}, J. Algebraic Comb., \textbf{33} (2011), 259--276.

\bibitem[CI20]{CI20}G.~Cerulli Irelli, \emph{Three lectures on quiver Grassmannians} in “Representation Theory and Beyond", Cont. Math. 758, ed. Jan Št'ovíček, Jan Trlifaj
American Mathematical Soc., 2020, 57--88.

\bibitem[CEFR21]{CEFR2018} G.~Cerulli Irelli, F.~Esposito, H.~Franzen, M.~Reineke, \emph{Cellular decomposition and algebraicity of cohomology for quiver Grassmannians}, Advances in Mathematics,
Volume~\textbf{379} (2021), Paper No. 107544.

\bibitem[CFR12]{CFR12} G.~Cerulli Irelli, E.~Feigin, M.~Reineke, \emph{Quiver Grassmannians and degenerate flag varieties}, Algebra Number Theory, \textbf{6} (2012), no. 1, 165--194.

\bibitem[CFR13]{CFR13} G.~Cerulli Irelli, E.~Feigin, M.~Reineke, \emph{Desingularisation  of quiver Grassmannians for Dynkin quivers}, Adv. Math. \textbf{245} (2013) 182--207.


\bibitem[CFFFR17]{CFFFR17} G.~Cerulli Irelli, X.~Fang, E.~Feigin, G.~Fourier, M.~Reineke, \emph{Linear degenerations of flag varieties}, Mathematische Zeitschrift, \textbf{287} (2017), no. 
1,  615--654.

\bibitem[CFFFR20]{CFFFR20} G.~Cerulli Irelli, X.~Fang, E.~Feigin, G.~Fourier, M.~Reineke, \emph{
Linear degenerations of flag varieties: partial flags, defining equations, and group actions}, Math. Z. 296 (2020), no. 1--2, 453--477.




\bibitem[CL15]{CL15} 
G.~Cerulli Irelli, M.~Lanini, \emph{ Degenerate flag varieties of type A and C are Schubert varieties}, \emph{Int. Math. Res. Not.} (2015), no.~15, 6353--6374. 




\bibitem[FF13]{FF13} 
E. Feigin, M. Finkelberg, \emph{Degenerate flag varieties of type A: Frobenius splitting and BW theorem}, Math. Z.,\textbf{275} (2013), no. 1--2, 55--77.

\bibitem[FFR17]{FFR17} 
E. Feigin, M. Finkelberg, M. Reineke, \emph{Degenerate affine Grassmannians and loop quivers}, Kyoto J. Math., Volume~\textbf{57} (2017), no. 2, 445--474.

\bibitem[FL06]{FL06} G.~Fourier, P.~Littelmann,
\emph{Tensor product structure of affine Demazure modules and limit constructions,}  Nagoya Math. J. 182 (2006), 171--198.


\bibitem[FLP22]{FLP21}
E. Feigin, M. Lanini, A. P\"utz, \emph{Totally nonnegative Grassmannians, Grassmann necklaces and quiver Grassmannians}, Canadian Journal of Mathematics (2022). \url{https://doi.org/10.4153/S0008414X22000232}

\bibitem[FLP23]{FLP23} E.~Feigin,  M.~Lanini, A.~P\"utz, \emph{Generalized juggling patterns, quiver Grassmannians and affine flag varieties},  \url{http://arxiv.org/abs/2302.00304}. 

\bibitem[GKM98]{GKM1998} M.~Goresky, R.~Kottwitz, R.~MacPherson, \emph{Equivariant cohomology, Koszul duality, and the localization theorem}, Invent. Math., \textbf{131} (1998), 25--83.

\bibitem[Ga01]{Ga01}
D. Gaitsgory, \emph{Construction of central elements in the affine Hecke algebra via nearby cycles}, Inventiones mathematicae (2001), Volume 144, Issue 2, pp 253--280.

\bibitem[G01]{G01}
U.~G\"ortz, \emph{On the flatness of models of certain Shimura varieties of PEL-type}, Math. Ann. 321 (2001), 689--727.

\bibitem[G03]{G03}
U.~G\"ortz, \emph{On the Flatness of local models for the symplectic group}, Advances in Mathematics 176 (2003), no. 1, 89--115.



\bibitem[H13]{H13}
X.~He, \emph{Normality and Cohen–Macaulayness of local models of Shimura Varieties}, Duke Mathematical Journal 162 (2013), no. 13, 2509--2523.

\bibitem[HN02]{HN02}
T.~Haines, B.C.~Ng\^o, \emph{Nearby cycles for local models of some Shimura varieties}, Compositio Math. 133 (2002), no. 2, 117--150.

\bibitem[HR20]{HR20}
T.~Haines, T.~Richarz, \emph{Smoothness of Schubert varieties in twisted affine Grassmannians}, Duke Math. J. 169 (2020), no. 17, 3223–3260.



\bibitem[HR23]{HR23}T.~Haines, T.~Richarz, \emph{Normality and Cohen-Macaulayness of parahoric local models}, J. Eur. Math. Soc. (JEMS) 25 (2023), no. 2, 703--729.


\bibitem[HZ23-1]{HZ23-1}
X.~He, N.~Zhang,  \emph{Degenerations of Grassmannians via lattice configurations}, International Mathematics Research Notices 2023 (2023), no. 1, 298--349.

\bibitem[HZ23-2]{HZ23-2}
X.~He, N.~Zhang,  \emph{Degenerations of Grassmannians via lattice configurations II,
} arXiv:2305.00158. 


\bibitem[Kac85]{Kac85}
V.~Kac, \emph{Infinite dimensional Lie algebras}, Cambridge Univ. Press, Cambridge, 1985.




\bibitem[KLS13]{KLS13} A.~Knutson, T.~Lam, D.E.~Speyer, \emph{Positroid varieties:  juggling and geometry}, Compos. Math., 149(10):1710--1752, 2013.


\bibitem[KS14]{KS14} B.~Keller, S.~Scherotzke, \emph{Desingularisations of quiver Grassmannians 
via graded quiver varieties}, Advances in Mathematics 256 (2014) 318--347.

\bibitem[Kum02]{Kum02}
S. Kumar, Kac–Moody Groups, Their Flag Varieties and Representation Theory. Progress in Mathematics, vol. 204 (Birkh\"auser, Boston, 2002)

\bibitem[Lam16]{Lam16}
T.~Lam, \emph{Totally nonnegative Grassmannian and Grassmann polytopes}, 
Current developments in mathematics 2014, 51–152, Int. Press, Somerville, MA, 2016.

\bibitem[LP20]{LaPu2020} M.~Lanini, A.~P\"utz, \emph{GKM-Theory for Torus Actions on Cyclic Quiver Grassmannians},  \url{http://arxiv.org/abs/2008.13138}, \emph{to appear in} Algebra Number Theory.

\bibitem[LP23]{LaPu2021}M.~Lanini, A.~P\"utz, \emph{Permutation actions on quiver Grassmannians for the equioriented cycle via GKM-theory}, J. Algebr. Comb. \textbf{57}, 915--956 (2023).


\bibitem[Lus98a]{Lus98a}
G.~Lusztig, \emph{Total positivity in partial flag manifolds}, Representation Theory 2 (1998) 70--78.



\bibitem[P18]{P18}
G.~Pappas, \emph{Arithmetic models for Shimura varieties}, Proceedings of the ICM -- Rio 2018. Vol. II.
Invited lectures, 377--398, World Sci. Publ., Hackensack, NJ, 2018.

\bibitem[PRS13]{PRS13}
G.~Pappas, M.~Rapoport, B.~Smithling, \emph{Local models of Shimura varieties},
I. Geometry and combinatorics, Handbook of moduli. Vol. III, Adv. Lect. Math.
(ALM), vol. 26, Int. Press, Somerville, MA, 2013, pp. 135--217.

\bibitem[PZ13]{PZ13}
 G.~Pappas, X.~Zhu, \emph{Local models of Shimura varieties and a conjecture of Kottwitz}, Invent. Math. 194 (2013), no. 1, 147--254.


\bibitem[PZ22]{PZ22}
 G.~Pappas, I.~Zachos, \emph{Regular integral models for Shimura varieties of orthogonal type}, Compos. Math. 158 (2022), no. 4, 831--867.

\bibitem[Pos06]{Pos06} A.~Postnikov, \emph{Total  positivity,   Grassmannians,   and  networks},  Preprint, \url{http://math.mit.edu/~apost/papers/tpgrass.pdf}, 2006.

\bibitem[Pue19]{Pue2019} A.~P\"utz, \emph{Degenerate Affine Flag Varieties and Quiver Grassmannians}, Dissertation, Ruhr-Universit\"at Bochum (2019). \url{https://doi.org/10.13154/294-6576} 

\bibitem[Pue22]{Pue2020} A.~P\"utz, \emph{Degenerate Affine Flag Varieties and Quiver Grassmannians}, Algebr. Represent Theor, \textbf{25} (2022), 91--119.

\bibitem[PR23]{PuRe2022} A.~P\"utz, M.~Reineke, \emph{Desingularizations of Quiver Grassmannians for the Equioriented Cycle Quiver}, \url{http://arxiv.org/abs/2302.05384}.


\bibitem[Re13]{Re13} M.~Reineke, \emph{Every projective variety is a quiver Grassmannian}, Algebras and Representation Theory, \textbf{ 16} (2013), 1313--1314.


\bibitem[Rie99]{Rie99}
K.~Rietsch, \emph{An algebraic cell decomposition of the nonnegative part of a flag variety},  
J. Algebra 213 (1999), no. 1, 144--154.





\bibitem[S17]{S17} S.~Scherotzke, \emph{Desingularisation of quiver Grassmannians via Nakajima
categories}, Algebras and Representation Theory 20, 231--243 (2017).

\bibitem[Sch14]{Schiffler2014} R.~Schiffler, \emph{Quiver Representations}, CMS Books in Mathematics, Springer-Verlag, Cham, Switzerland, 2014. 


\bibitem[Scho92]{Scho92}A.~Schofield, \emph{General representations of quivers}, Proc. London Math. Soc. (3) 65 (1992), no. 1, 46--64.

\bibitem[W05]{W05}
L.~Williams, \emph{Enumeration of totally positive Grassmann cells}, 
Adv. Math. 190 (2005), no. 2, 319--342. 

\bibitem[Zho19]{Zho19}
Q.~Zhou, \emph{Convex polytopes for the central degeneration of the affine Grassmannian}, Adv. Math. 348 (2019), 541--582.


\bibitem[Zhu17]{Zhu17}
X.~Zhu, \emph{An introduction to affine Grassmannians and the geometric Satake equivalence}, Geometry of
moduli spaces and representation theory, 59--154, IAS/Park City Math. Ser., 24, Amer. Math. Soc., Providence, RI, 2017.


\bibitem[Zhu19]{Zhu19}
X.~Zhu, \emph{On the coherence conjecture of Pappas and Rapoport}, Ann. of Math. (2) 180 (2014), no. 1, 1--85.


\end{thebibliography}
\end{document}